\documentclass[12pt]{amsart}
\usepackage{amssymb}
\usepackage{stmaryrd}
\usepackage{mathptmx}
\usepackage[all]{xy}

\numberwithin{equation}{section}

\newcommand{\Z}{\mathbb{Z}}
\newcommand{\R}{\mathbb{R}}

\newcommand{\N}{\mathbb{N}}
\newcommand{\Q}{\mathbb{Q}}

\newcommand{\Ic}{\mathcal{I}}
\newcommand{\Jc}{\mathcal{J}}

\newcommand{\Sc}{\mathcal{S}}

\newcommand{\be}{\mathbf{e}}

\DeclareMathOperator{\id}{id}

\DeclareMathOperator{\ini}{in}
\DeclareMathOperator{\Inc}{Inc (\N)}
\DeclareMathOperator{\ind}{ind}

\DeclareMathOperator{\Sym}{Sym}

\DeclareMathOperator{\pnt}{\raise 0.5mm \hbox{\large\bf.}}

\newcommand{\la}{\langle}
\newcommand{\ra}{\rangle}

\let\phi=\varphi

\newtheorem{thm}{\bf Theorem}[section]
\newtheorem{lem}[thm]{\bf Lemma}
\newtheorem{cor}[thm]{\bf Corollary}
\newtheorem{prop}[thm]{\bf Proposition}

\theoremstyle{definition}
\newtheorem{defn}[thm]{\bf Definition}

\newtheorem{rem}[thm]{\bf Remark}
\newtheorem{ex}[thm]{\bf Example}

\title{Equivariant Hilbert Series in non-Noetherian Polynomial Rings}

\author[Uwe Nagel]{Uwe Nagel}
\address{Department of Mathematics, University of Kentucky, 715 Patterson Office Tower, Lexington, KY 40506-0027, USA}
\email{uwe.nagel@uky.edu}

\author{Tim R\"omer}
\address{Universit\"at Osnabr\"uck, Institut f\"ur Mathematik, 49069 Osnabr\"uck, Germany}
\email{troemer@uos.de}

\begin{document}

\begin{abstract} We introduce and study equivariant Hilbert series of ideals in polynomial rings in countably many variables that are invariant under a suitable action of a symmetric group or the monoid $\Inc$ of strictly increasing functions.  Our first main result states that these series are rational functions in two variables.  A key  is to  introduce also
suitable submonoids of $\Inc$ and to compare invariant filtrations induced by their actions. Extending a result by Hillar and Sullivant, we show that any ideal that is invariant under these submonoids admits a Gr\"obner basis consisting of finitely many orbits.
As our second main result  we prove that the Krull dimension and multiplicity of ideals in an invariant filtration grow eventually linearly and exponentially, respectively, and we determine the terms that dominate this growth.
\end{abstract}

\thanks{The first author was partially supported by Simons Foundation grant \#317096. \\
Both authors are grateful to Serkan Hosten for an inspiring discussion. }

\maketitle



\section{Introduction}
\label{sect_intro}

Recent results in algebraic statistics such as the Independent Set Theorem \cite[Theorem 4.7]{HS} as well as challenging  problems on  families  of varieties, tensors, or representations in spaces of increasing dimension  motivate the
study of  ideals in  a polynomial ring $K[X]$ in a countable set of infinite variables over a field $K$ (see, e.g., \cite{Draisma-factor, DK,  Draisma, HS, SS-14, S}).
Often these ideals are invariant under an action of a symmetric group  or submonoids of the monoid of strictly increasing functions on the set $\N$ of positive integers
\[
\Inc = \{\pi: \N \to \N \; | \; \pi (i) < \pi (i+1) \text{ for all } i \ge 1\}
\]
(see, e.g., \cite{AH, HS}).
Any such ideal can be described as a union of  ideals $I_n$ that form an invariant filtration, where each $I_n$ is an ideal in a polynomial ring $K[X_n] \subset K[X]$ in finitely many variables (see Section \ref{sec:filt} for details). These ideals give rise to an ascending chain of ideals in $K[X]$
\[
\la I_1\ra_ {K[X] } \subset \la I_2\ra_{K[X]} \subset \cdots .
\]
A key of our approach is to introduce submonoids of $\Inc$ that fix an initial segment
\[
\Inc^i = \{\pi \in \Inc \; | \;  \pi (j) = j \text{ if } j \le i\},
\]
where $i \ge 0$ is any integer. Note that $\Inc^0 = \Inc$.  Extending a central result by Hillar and Sullivant in \cite{HS} we show that each $\Inc^i$-invariant ideal $I$ admits a Gr\"obner basis that consists of   $\Inc^i$-orbits of finitely many polynomials (see Theorem \ref{thm:i-finiteness}). In particular, the ideal $I$ is generated by finitely many such orbits. Notice that any ideal  of $K[X]$ that is invariant under a suitable action of a symmetric group $\Sym (\infty)$ is also an $\Inc^i$-invariant ideal. However, the converse is not true (see Example \ref{exa:inc but not sym-invariant}).

Recall that, for a polynomial ring $P$ in finitely many variables, the Hilbert series $H_M (t)$ of a finitely generated graded $P$-module $M$ allows one to introduce and to study invariants of $M$ such as its Krull dimension and its multiplicity. The Hilbert series is a formal power series that is in fact a rational function.

In order to simulate this classical  approach for an $\Inc^i$-invariant ideal $I$ we consider an $\Inc^i$-invariant filtration $\Ic = (I_n)_{n \in \N}$ of ideals $I_n$ that describes $I$ and study the Hilbert series of all the ideals $I_n$ at once by introducing a formal power series
in two variables $s$ and $t$
\[
H_{\Ic} (s, t) = \sum_{n \ge 0} H_{K[X_n]/I_n} (t) \cdot s^n = \sum_{n \ge 0,\, j \ge 0} \dim_K [K[X_n]/I_n]_j \cdot s^n t^j.
\]
We call this the \emph{equivariant} or \emph{bigraded Hilbert series} of the filtration $\Ic$ and show that it is in fact a rational function of a certain form (see Proposition \ref{prop:main thm} and Theorem \ref{thm:main thm sym}). Note that this result is true regardless of the characteristic of the base field.

The rationality of Hilbert series of graded modules over polynomial rings in finitely many variables can be shown, for example,  using the finiteness of free resolutions or by induction on the Krull dimension. In order to establish the above rationality results for infinitely many variables we adapt the second approach.
However, an immediate difficulty is that, given an $\Inc$-invariant or an $\Sym (\infty)$-invariant  ideal $I$ and a linear form $\ell$ of $K[X]$, the ideal $\langle I, \ell\rangle$ is typically no longer invariant. A key to our argument and one of the reasons for introducing the monoids $\Inc^i$ is the observation that the ideal $\langle I, \ell\rangle$ is $\Inc^i$-invariant for suitably large integers $i$.

Analyzing properties of invariants by studying them asymptotically has become a systematic approach in commutative algebra only in the last decades (see, e.g., \cite{BoH, BMMNZ, CHT, EL,EEL, K}).
The mentioned rationality result allows us to asymptotically determine the Krull dimensions and degrees of the ideals $I_n$ in an invariant filtration. In fact, we show that the dimension of $K[X_n]/I_n$ is a linear function in $n$ for sufficiently large $n$ and that the degree of $I_n$ grows exponentially in $n$. More precisely, we prove that there are non-negative integers $M, L$ such that the  limit  of
\[
\frac{\deg I_n}{M^{n } \cdot  n^{L}}
\]
 as $n \to \infty$ exists and is equal to a positive rational number (see Theorem \ref{thm:asymp dim}). It is tempting to consider this limit and the leading coefficient of the mentioned linear function as the degree of $I$ and the Krull dimension of $K[X]/I$, respectively.

This manuscript is organized as follows. In Section 2 we recall some  results on Hilbert functions and introduce  notation. The finiteness of Gr\"obner bases of $\Inc^i$-invariant ideals of $K[X]$ is established in Section 3. Certain technical results that are useful for studying $\Inc^i$-invariant ideals are derived in Section 4. Invariant filtrations are studied in Section 5. In particular, it is shown that they stabilize (see Corollary \ref{cor:filtration stable}). Section 6 is devoted to proving the rationality of bigraded Hilbert series in the case of $\Inc^i$-invariant chains of monomial ideals. In this case, we also derive more detailed information on the form of the rational function (see Theorem \ref{thm:rat Hilb series, rev}). Combining this with the Gr\"obner basis result of Section 2, we obtain the desired rationality result for any invariant chain in Section 7. There we also conduct our asymptotic study of invariants of ideals in such a chain.  Various examples illustrate our results.


\section{Preliminaries}

Here we recall some basic facts and introduce notation used throughout this work.

Let $P$ be a polynomial ring in finitely many variables over any field $K$. The \emph{Hilbert function} of a finitely generated graded $P$-module  $M$ in degree $j$ is  $h_{M}(j)=\dim_K [M]_j$, where we denote by $[M]_j$ the degree $j$ component of $M$.  It is well-known that,  for large $j$, this is actually a polynomial function in $j$. Equivalently, the Hilbert series of $M$ is a rational function. Recall that the \emph{Hilbert series} of $M$ is the formal power series
\[
H_M (t) = \sum_{j \in \Z} h_M (j) \cdot t^j.
\]
In fact, it can be written as $\frac{g(t)}{(1-t)^{\dim M}}$, where $g(t) \in \Z[t]$ and $g(1)$ is the  \emph{multiplicity} of $M$. It is positive if  $M \neq 0$.  If $M = P/I$ for some graded ideal $I$ of $P$, then we refer to $g(1) = \deg I$ also as the \emph{degree} of $I$.

Throughout we use $\N$ and $\N_0$ to  denote the set of positive and non-negative integers, respectively. For any $k \in \N$, set $[k] = \{1,2,\ldots,k\}$. Sometimes it is convenient to write $[0]$ for the empty set.

We  study $\Inc^i$-invariant ideals in  the following setting.
Fix some $c \in \N$, and put, for each $n \in \N$,
\[
X_n = \{x_{i, j} \; | \; \; i \in [c], \ j \in [n]\}.
\]
Thus, for each $n \ge 2$, \ $X_n \setminus X_{n-1} = \{x_{1,n},x_{2, n},\ldots,x_{c,n}\}$. Set $X=\bigcup_{n\geq 1} X_n$.  Denote by $K[X_n]$ and
$K[X]$ the polynomial rings over  $K$ in the variables in $X_n$ and $X$, respectively.  Thus, for any positive integer $n$, there is a natural embedding $\iota_n: K[X_n] \to K[X]$ and a natural projection $\rho_m: K[X] \to K[X_n]$ such that $\rho_n \circ \iota_n = \id_{K[X_n]}$. Hence, each $K[X_n]$ is a retract of $K[X]$.

As mentioned above, we will consider ideals $I_n \subset K[X_n]$ that induce an ascending chain in $K[X]$
\[
\la I_1\ra_ {K[X] } \subset \la I_2\ra_{K[X]} \subset \cdots .
\]
We study the Hilbert series of the quotient rings $K[X_n]/I_n$ simultaneously by considering a formal power series in two variables
\[
\sum_{n \ge 0,\, j \ge 0} \dim_K [K[X_n]/I_n]_j \cdot s^n t^j.
\]
Much of this work is devoted to showing that this series is a rational function under suitable assumptions. This is also related to work by Sam and Snowden (see, e.g., \cite{SS-12b, SS-14, S}).


\section{$\Inc^i$-equivariant Gr\"obner Bases}

The main result of this section is that  every $\Inc^i$-invariant ideal of $K[X]$ has a finite $\Inc^i$-equivariant Gr\"obner basis. This extends one of the main results by Hillar and Sullivant  \cite[Theorem 3.1]{HS}.

We need some preparation.
Let $S$ be any set. A \emph{well-partial-order} on $S$ is a partial order $\le$ such that for any infinite sequence $s_1,s_2, \ldots$ of elements in $S$ there is a pair of indices $i < j$ such that $s_i \le s_j$.

\begin{rem}
    \label{rem:well-partial-order}
(i) If $S$ and $T$ are sets which have  well-partial-orders, then it is well known that their Cartesian product  $S \times T$ also admits a well-partial order, namely the componentwise partial order defined by  $(s,t) \le (s', t')$ if $s \le s'$ and $t \le t'$. The analogous statement is true for finite products.  In particular, it follows that the componentwise partial order on $\N_0^c$ is a well-partial-order, a fact which is also called Dickson's Lemma.

(ii) Let $S$ and $T$ be sets such that there union admits a partial order with the  property that its restrictions to $S$ and $T$   are well-partial-orders. Then it follows immediately that  the given order on $S \cup T$ is also a well-partial-order.  The analogous statement is true for finite unions.

(iii) Given a set $S$ with a partial order $\le$, define a partial order on the set $S^*$ of finite sequences of elements in $S$ by
$(s_1,\ldots,s_p) \le_H (s'_1,\ldots,s'_q)$ if there is a strictly increasing map $\phi: [p] \to [q]$ such that $s_i \le s'_{\phi(i)}$ for all $i \in [p]$. This order is called the \emph{Higman order} on $S^*$. It is a well-partial-order by Higman's Lemma (see \cite{H} or, for example,  \cite{Draisma}).
\end{rem}

We now define  partial orders on $S^*$ that are coarser than the Higman order.

\begin{defn}
Let $S$ be a set with a partial order $\le$. For $i \in \N$, define a partial order $\le_i$ on $S^*$ by
$(s_1,\ldots,s_p) \le_i (s'_1,\ldots,s'_q)$ if there is a strictly increasing map $\phi: [p] \to [q]$ satisfying $\phi (j) = j$ if $j \in [p] \cap [i]$ and   $s_j \le s'_{\phi(j)}$ for all $j \in [p]$. Furthermore, we denote by $\le_0$  the Higman order $\le_H$.
\end{defn}

We are not a aware of a reference for the following result. Thus, we include a proof for the convenience of the reader.

\begin{lem}
   \label{lem:i-order}
For each non-negative integer $i$, the order $\le_i$ is a well-partial-order on $S^*$.
\end{lem}

\begin{proof}
By Higman's Lemma it suffices to discuss $ i > 0$. Note that
\[
S^* = \bigcup_{j = 1}^{i} S^j \cup T,
\]
where $T = S^i \times S^*$ is the set of sequences whose length is at least $i+1$.  Observe that the restriction of $\le_i$ on $S^*$ to any $S^j$ with $j \le i$ is the componentwise partial order on $S^j$, which is a well-partial-order by Remark \ref{rem:well-partial-order} (i). Furthermore, for any two elements $t = (s_1, s_2), t'  = (s'_1, s'_2)\in T$ with $s_1, s'_1 \in S^i$ and $s_2, s'_2 \in S^*$, one has $t \le_i t'$ if and only if $s_1 \le s'_1$ in the componentwise partial order on $S^i$ and $s_2 \le_H s'_2$ in the Higman order on $S^*$. Hence the restriction of $\le_i$ on $S^*$ to $T$ is a well-partial-order by Remark \ref{rem:well-partial-order} (i), (iii). We conclude that $\le_i$ is a well-partial-oder on $S^*$ by  Remark \ref{rem:well-partial-order} (ii).
\end{proof}

We now define submonoids of $\Inc$ that fix an initial segment of variables:

\begin{defn}
For any  $i \in \N_0$, set
\[
\Inc^i = \{\pi \in \Inc \; | \;  \pi (j) = j \text{ if } j \le i\}.
\]
Note that, in particular,  $\Inc^0 = \Inc$.
\end{defn}

We always assume that the action of any  $\Inc^i$ on $K[X]$ is induced by
\[
\pi \cdot x_{j, k} = x_{j, \pi (k)},
\]
and refer to this as the \emph{standard action} of $\Inc^i$. Thus, the \emph{$\Inc^i$-orbit} of a polynomial $f \in K[X]$ consists of all the elements $\pi (f)$, where $\pi \in \Inc^i$. Note that the action of any $\pi \in \Inc^i$ gives an injective $K$-algebra homomorphism $K[X] \to K[X]$. An ideal $I$ of $K[X]$ is \emph{$\Inc^i$-invariant} if $\pi (f)$ is in $I$ whenever $f \in I$ and $\pi \in \Inc^i$.

We want to show finiteness properties of $\Inc^i$-Gr\"obner bases. We denote the initial monomial of  a polynomial $f \in K[X]$ with respect to a  monomial order $\preceq$ on $K[X]$ by $\ini_{\preceq} (f)$. Let $I \subset K[X]$ be an $\Inc^i$-invariant ideal.  An \emph{$\Inc^i$-Gr\"obner basis} of $I$ with respect to  $\preceq$  is a subset $B$ of $I$ such that for each $f \in I$ there is a $g \in B$ and a $\pi \in \Inc^i$ such that $\ini_{\preceq} (\pi(g))$  divides  $\ini_{\preceq} (f)$. Usually, we require additionally that $\pi (u) \preceq \pi (v)$ whenever  $\pi \in \Inc^i$ and $u, v \in K[X]$ are monomials with $u \preceq v$. This assumption implies in particular, for any $g \in K[X]$,
\[
\ini_{\preceq} \left (\pi(g) \right ) = \pi \left ( \ini_{\preceq} (g) \right )
.\]

We are ready for an extension of one of the main results  by Hillar and Sullivant \cite[Theorem 3.1]{HS}  from $i = 0$ to arbitrary $i \ge 0$.

\begin{thm}
     \label{thm:i-finiteness}
Fix $i \in \N_0$ and consider a monomial order $\preceq$ on $K[X]$ satisfying $\pi (u) \preceq \pi (v)$ whenever  $\pi \in \Inc^i$ and $u, v \in K[X]$ are monomials with $u \preceq v$. Then any $\Inc^i$-invariant ideal  admits a finite $\Inc^i$-Gr\"obner basis with respect to $\preceq$.
\end{thm}

Note that the assumption on the monomial order $\preceq$ is satisfied by the lexicographic order induced by the ordering of the variables $x_{i, j} \preceq x_{i', j'}$ if either $i < i'$ or  $i = i'$ and $j < j'$.

\begin{cor}
    \label{cor:Inc-noeth}
For each $i \in \N_0$, the ring $K[X]$ is $\Inc^i$-noetherian, that is, any $\Inc^i$-invariant ideal of $K[X]$ is generated by finitely many $\Inc^i$-orbits.
\end{cor}

\begin{proof}
This follows from Theorem \ref{thm:i-finiteness} because the union of the $\Inc^i$-orbits of an $\Inc^i$-Gr\"obner basis of an ideal $I$  generate the ideal (see \cite[Proposition 2.10]{HS}).
\end{proof}

\begin{proof}[Proof of Theorem \ref{thm:i-finiteness}] We follow the ideas of  \cite[Theorem 3.1]{HS} (see also  \cite[Theorem 2.3]{Draisma} or \cite{AH}).

Fix $ i \ge 0$ and consider the following partial order on the set of monomials in $K[X]$  defined by $u |_{\Inc^i}\, v$ if there is some $\pi \in \Inc^i$ such that $\pi (u)$ divides $v$.  According to \cite[Theorem 2.12]{HS} (see also \cite[Proposition 2.2]{Draisma}) it suffices to show that $|_{\Inc^i}$ is a well-partial-order.

Let $S = \N_0^c$ and associate to a monomial $u \in K[X]$ a sequence $s(u) = (s_1,\ldots, s_p)$ in $S^*$, where the $m$-th entry of $s_n \in \N_0^c$ is the exponent of the variable $x_{m, n}$ in $u$ and $p$ is the largest column index of a variable that divides $u$.

Consider now an infinite sequence $u_1, u_2, \ldots$ of monomials in $K[X]$. It induces an infinite sequence $s(u_1), s(u_2), \ldots$ of elements in $S^*$. Since $\le_i$ is a well-partial-order on $S^*$ by Lemma \ref{lem:i-order}, there are indices $k <\ell$ such that $s(u_k) = (s_1,\ldots, s_p)  \le_i s(u_{\ell}) = (t_1,\ldots, t_q)$.
Thus, there is a strictly increasing map $\phi: [p] \to [q]$ satisfying $\phi (n) = n$ if $n \in [p] \cap [i]$ and   $s_n \le t_{\phi(n)}$ for all $n \in [p]$. It is clear that there is some $\pi \in \Inc^i$ such that $\pi (n) = \phi (n)$ for all $n \in [p]$. (Note that there is more than one choice for $\pi$.)

We claim that $\pi (u_k)$ divides $u_{\ell}$.  Indeed, note that the column indices of variables dividing $\pi (u_k)$ are contained in  $\pi ([p])$. Thus, to check the divisibility condition it suffices to consider variables whose column indices are in $\pi ([p])$. Now,  if $n \in \pi ([p]) = \phi([p])$, then $s_{\phi^{-1} (n)} \le t_n$, which concludes the argument.
\end{proof}

Note that Corollary \ref{cor:Inc-noeth} cannot be extended to any action of $\Inc$ on $K[X]$, even if it induces an embedding of $\Inc$ into the ring of $K$-algebra endomorphisms on $K[X]$.

\begin{ex}
     \label{exa:counterexa}
Let $c = 2$, and define a non-standard action of $\Inc$ on $K[X]$ by
\[
\pi (x_{j, k}) = \begin{cases}
x_{j, \pi (k)} & \text{ if } j = 1, \\
x_{j, k} & \text{ if } j = 2.
\end{cases}
\]

Consider an  ideal
\[
I =  \la x_{1, j} \cdot  x_{2, 1 + 2 k} \; | \; j \in \N, \ k \in \N_0 \ra \subset K[X].
\]
Observe that $I$ is invariant under the above non-standard action of $\Inc$, but it is not invariant under the standard action of $\Inc$.
However, the ideal $I$ cannot be generated by finitely many orbits under the non-standard $\Inc$-action.
\end{ex}


\section{Decompositions}
\label{secLdecomp}

Here we establish various decompositions of elements and submonoids of $\Inc$ that are needed in subsequent sections. The reader may skip this part at first reading and proceed right away to Section \ref{sec:filt}.

We will frequently use the following particular elements $\sigma_i \in \Inc^i$, $i \ge 0$,  defined by
\[
\sigma_i (j) = \begin{cases}
j  & \text{if } 1 \le j \le i,  \\
j+1 & \text{if } i < j.
\end{cases}
\]
We call $\sigma_i$ the \emph{$i$-shift}. Note that $\sigma_i \in \Inc^j$ if $ j \le i$.   A straightforward computation gives the following useful observation.

\begin{lem}
     \label{lem:element comp}
Given $i \in \N_0$ and  $\pi \in \Inc^i$,
there is an element $\tau \in \Inc^{i+1}$ such that
\[
\sigma_i \circ \pi =  \tau \circ \sigma_i.
\]
Moreover, if $\pi (m) \le n$, then $\tau (m+1) \le n+1$.
\end{lem}

\begin{proof}
If $  i >0$, then the claim is true for $\tau$ defined by
\[
\tau (j) = \begin{cases}
j  & \text{if } 1 \le j \le i, \\
\pi (j-1)+1 & \text{if } i+1 \le j.
\end{cases}
\]
In $i =0$, then  $\tau$ can be taken as the element defined by
\[
\tau (j) = \begin{cases}
1 & \text{if }  j = 1, \\
\pi (j-1)+1 & \text{if } 2 \le j.
\end{cases}
\]
One checks that in both cases the element $\tau$ is well-defined, that is, it is indeed in  $\Inc^{i+1}$, and it satisfies the desired equation as well as the stated additional property.
\end{proof}

It is worth singling out the following special case.

\begin{cor}
    \label{cor:sigma comp}
If $j > i  \ge 0$, then
$\sigma_i \circ \sigma_{j-1} = \sigma_j \circ \sigma_i$.
\end{cor}

\begin{proof}
We apply Lemma \ref{lem:element comp} to $\pi = \sigma_{j -1} \in \Inc^i$. Then,  the element $\tau$ defined in the proof of Lemma \ref{lem:element comp} equals $\sigma_j$, and we are done.
\end{proof}

\begin{rem}
Notice that  Corollary \ref{cor:sigma comp} is false if $j \le i$. Furthermore, considering $\tau (i)$ one checks that there is no $\tau \in \Inc$ such that
$\sigma_i \circ \tau = \sigma_{i-1} \circ \sigma_i$.
\end{rem}

We also need the following fact.

\begin{lem}
    \label{lem:element decomp}
For each $\pi \in \Inc^i \setminus \Inc^{i+1}$, there is some $\tau \in \Inc^{i+1}$ such that
\[
\pi = \tau \circ \sigma_i.
\]
\end{lem}

\begin{proof}
Since $\pi$ is not in $\Inc^{i+1}$, the function
 $\tau$ defined by
\[
\tau (j)  = \begin{cases}
j & \text{if } 1 \le j \le i+1,\\
\pi(j-1)  & \text{if } i+2 \le j
\end{cases}
\]
is increasing. It is straightforward to check the desired equality.
\end{proof}

Each monoid $\Inc^j$ can naturally be filtered by  suitable subsets.

\begin{defn} For integers $i \ge 0, m \le n$, set
\[
\Inc^i_{m, n} = \{ \pi \in \Inc^i \; | \; \pi (m ) \le n\}.
\]
\end{defn}

Thus, each $\pi \in \Inc^i_{m, n}$ induces an embedding $ [m] \to [n]$.
 We define  sets
 \[
 \Inc^{j}_{m, n} \circ \Inc^i_{k, m} = \{ \tau \circ \pi \; | \; \tau \in  \Inc^{j}_{m, n}  \text{ and } \pi \in \Inc^{i}_{k, m} \}
 \]
  whenever $k \le m \le n$ and $i, j \ge 0$. The following observation turns out to be useful.

\begin{prop}
    \label{prop:inc-decomp}
Consider integers $i, m, n$ such that $i \ge 0$ and $n > m \ge 1$.     Then there is the following decomposition, as subsets of $\Inc$:
\begin{equation}
\label{eq:decomp}
\Inc^i_{m,n} = \Inc^{i+1}_{m+1, n} \circ \Inc^i_{m, m+1}.
\end{equation}
In particular,
\[
\Inc^i_{m,n} = \Inc^{i}_{m+1, n} \circ \Inc^i_{m, m+1}.
\]
\end{prop}

\begin{proof}
The right-hand side of Equation \eqref{eq:decomp} is clearly contained in the left-hand side. For the reverse inclusion, consider any $\pi \in \Inc^i_{m, n}$. If $\pi$ is the identity, then $\pi = \pi \circ \pi$ implies the desired containment. If $\pi$ is not the identity, then there is some integer $j \ge i$ such that $\pi \in \Inc^j \setminus \Inc^{j+1}$. Thus,  Lemma \ref{lem:element decomp} (see its proof) gives an element $\tau \in \Inc^{j+1}_{m+1, n} \subset  \Inc^{i+1}_{m+1, n}$ such that $\pi = \tau \circ \sigma_j$. Since $\sigma_j$ is in $\Inc^i_{m, m+1}$, the Equation \eqref{eq:decomp} follows.

The final assertion is a consequence of the inclusion $\Inc^{i+1}_{m, n} \subset \Inc^{i}_{m, n}$.
\end{proof}


\section{Invariant Filtrations}
  \label{sec:filt}

A key to our study of numerical invariants of $\Inc^i$-invariant ideals are filtrations. We introduce them here. As  another  application of Theorem \ref{thm:i-finiteness},  we show that these chains  stabilize.

\begin{defn}
    \label{def:filtration}
(i) An  \emph{$\Inc^i$-invariant filtration} (also called $\Inc^i$-invariant chain)  is a sequence $\Ic = (I_n)_{n \in \N}$ of ideals $I _n\subset K[X_n]$ such that, as subsets of $K[X]$, one has
\[
\Inc^i_{m, n} (I_m) \subset I_n \text{ whenever } m \le n.
\]
(ii) An $\Inc^i$-invariant  filtration $\Ic = (I_n)_{n \in \N}$ \emph{stabilizes}  if there is an integer $r$ such that, as ideals of $K[X_n]$, one has
\[
\la \Inc^i_{r, n} (I_r)\ra_{K[X_n]} = I_n \text{ whenever } r  \le n.
\]
In this case, the least integer $r \ge 1$ with this property is said to be the \emph{$i$-stability index} $\ind^i (\Ic)$ of $\Ic$, that is,
\[
\ind^i (\Ic) = \inf \{ r \; | \; \la \Inc^i_{m, n} (I_m)\ra_{K[X_n]} = I_n \text{ whenever } r \le m \le n\}.
\]
 \end{defn}

 Since the identity is in $\Inc^i_{m, n}$, an $\Inc^i$-invariant  filtration satisfies, $I_m \subseteq I_n$ (as subsets of $K[X]$) if $m \le n$.
  Observe  that each $\Inc^i$-invariant filtration $\Ic$ is also $\Inc^{i+1}$-invariant. Moreover, $\ind^i (\Ic) \ge \ind^{i+1} (\Ic)$.    If $i = 0$ we speak of an $\Inc$-invariant filtration, set $\ind (\Ic) =  \ind^0 (\Ic)$, and  call this number  the \emph{stability index} of $\Ic$.

 To simplify notation we do not mention the extension ring explicitly if it is clear from context. For example, we often write $\la \Inc^i_{m, n} (I_m)\ra$ instead of $\la \Inc^i_{m, n} (I_m)\ra_{K[X_n]}$.

 We now provide several characterizations of stability.

 \begin{lem}
     \label{lem:stab}
 Let $\Ic = (I_n)_{n \in \N}$  be an  $\Inc^i$-invariant  filtration. For a positive integer $r$, the following conditions are equivalent:
 \begin{itemize}

 \item[(a)]  $\Ic$ stabilizes and its $i$-stability index is at most $r$;

 \item[(b)]  For each pair of  integers $n \ge m \ge r$, one has
 \[
 \la \Inc^i_{m, n} (I_m)\ra_{K[X_n]} = I_n;
 \]

 \item[(c)] For each integer $n \ge r$, one has, as ideals of $K[X_n]$,
 \[
\bigcup_{j \le r} \la \Inc^i_{j , n} (I_j) \ra =  I_n.
 \]

 \end{itemize}
 \end{lem}

The lemma above shows that Definition \ref{def:filtration}(ii)  is equivalent to \cite[Definition 2.17]{HS} in the case of $\Inc^i$, which corresponds to condition (c) of  Lemma \ref{lem:stab}.

 \begin{proof}[Proof of Lemma \ref{lem:stab}]
 The implications (b) $\Rightarrow$ (a) and (a) $\Rightarrow$ (c) are clear. Note that Proposition \ref{prop:inc-decomp} gives for any pair of positive integers $m < n$
 \[
 \la \Inc^i_{m , n} (I_m) \ra = \la \Inc^i_{m+1 , n}  \circ \Inc^i_{m , m+1}  (I_m) \ra \subseteq  \la \Inc^i_{m+1 , n} (I_{m+1}) \ra.
 \]
 It follows that $\bigcup_{j \le r} \la \Inc^i_{j , n} (I_j) \ra = \la \Inc^i_{r , n} (I_j) \ra$, and thus (c) implies (a). Similarly, assuming (a), one gets if $r \le m \le  n$,
 \[
 I_n =  \la \Inc^i_{r, n} (I_r)\ra  \subseteq  \la \Inc^i_{m , n} (I_{m}) \ra \subseteq I_n.
 \]
Now  (b) follows.
 \end{proof}

Note that the extension ideals  in $K[X]$ of the ideals in an $\Inc^i$-invariant filtration $\Ic = (I_n)_{n \in \N}$ form an ascending chain of ideals
\[
\la I_1\ra_ {K[X] } \subset \la I_2\ra_{K[X]} \subset \cdots .
\]
In general, these ideals are not $\Inc^i$-invariant. However, their union
\[
 \bigcup_{n \ge 1} \la I_n\ra_{K[X]} = \bigcup_{n \ge 1} I_n
\]
is an $\Inc^i$-invariant ideal. The equality is easily seen,  and the claimed  invariance follows from the following observation: If $f \in K[X_m]$ and $\pi \in \Inc^i$, then, setting $n = \pi (m)$, one has   $\pi (f) \in K[X_n]$  and $\pi \in \Inc^i_{m, n}$ because $\pi$ is an increasing function.

We are ready to present the announced application of Theorem \ref{thm:i-finiteness}, which extends \cite[Theorem 3.6]{HS}.

\begin{cor}
     \label{cor:filtration stable}
Each $\Inc^i$-ascending chain stabilizes.
\end{cor}

\begin{proof} By the proof of  Theorem \ref{thm:i-finiteness}, divisibility  $|_{\Inc^i}$ gives a well-partial-order. Moreover, as mentioned below Lemma \ref{lem:stab} the concepts of stability in \cite{HS} and Definition \ref{def:filtration} are equivalent.  Now one argues as
 in the proof of \cite[Theorem 3.6]{HS}, which shows the claim for $i = 0$.
\end{proof}

\begin{rem}
   \label{rem:index est}
For an  $\Inc^i$-invariant filtration $\Ic = (I_n)_{n \in \N}$, set $I = \bigcup_{n \in \N} I_n$. Let $B$ be a finite $\Inc^i$-Gr\"obner basis of $I$, and denote by $r$ the least integer $n$ such that $B \subset K[X_{n}]$. Then the proof of \cite[Theorem 3.6]{HS} (see in particular the use of \cite[Lemma 2.18]{HS}) shows that $\ind^i (\Ic) \le r$.
\end{rem}


We say that two $\Inc^i$-invariant filtrations  $(I_n)_{n \in \N}$ and $(J_n)_{n \in \N}$ are \emph{equivalent filtrations}  if $\bigcup_{n \ge 1} I_n = \bigcup_{n \ge 1} J_n$. Among equivalent filtrations there is a unique maximal filtration. Indeed, given any $\Inc^i$-invariant ideal $I$ of $K[X]$, define a sequence of ideals $(I \cap K[X_n])_{n  \in \N}$. It is an $\Inc^i$-invariant filtration, which we call the  \emph{saturated filtration of $I$}. If now $(I_n)_{n \in \N}$  is an arbitrary $\Inc^i$-invariant filtration, then it is a subfiltration of the saturated filtration of $I = \bigcup_{n \ge 1} I_n$.

Given an ideal $I_r \subset K[X_r]$,
\[
\bigcap_{I_r \subseteq J \subset K[X] \text{ $\Inc^i$-invariant ideal} } J  = \la \Inc^i (I_r)\ra_{K[X]}
\]
is the smallest $\Inc^i$-invariant ideal that contains $I_r$.

Thus, we call it the  \emph{$\Inc^i$-closure} of $I_r$.
There are several   $\Inc^i$-invariant filtrations that are smaller, but equivalent to the saturated filtration of the $\Inc^i$-closure. One of them is given below, with two a priori different descriptions.

\begin{lem}
    \label{lem:equivfilt}
Given an integer $i \ge 0$ and an ideal $0 \neq \tilde{I} \in K[X_r]$, define two sequences of ideals $\Ic = (I_n)_{n \in \N}$  and $\Jc  = (J_n)_{n \in \N}$   by
\[
I_n = \begin{cases}
\la 0 \ra & \text{if } 1 \le n < r,\\
\tilde{I} & \text{if } n = r, \\
\la \Inc^i_{n-1,n} (I_{n-1}) \ra & \text{if } n > r
\end{cases}
\]
and
\[
J_n = \begin{cases}
\la 0 \ra & \text{if } 1 \le n < r,\\
\la \Inc^i_{r,n} (\tilde{I})\ra & \text{if } n \ge r.
\end{cases}
\]
Then $\Ic$  and $\Jc$ are equal $\Inc^i$-invariant filtrations with stability index $r$. Moreover, $J = \bigcup_{n \in \N} J_n$ is the $\Inc^i$-closure of $\tilde{I}$.
\end{lem}

\begin{proof}
Since
 $\Inc^i_{m,n} (J_m) \subset J_n$ whenever $n \ge  m \ge r$ the sequence $(J_n)_n$ is an $\Inc^i$-invariant filtration.
Applying Proposition \ref{prop:inc-decomp} repeatedly, we get, for $n > r$,
\[
J_n = \la \Inc^i_{r,n} (\tilde{I})\ra = \la \Inc^i_{n-1,n} \circ \Inc^i_{n-2,n-1} \circ \cdots \circ \Inc^i_{r+1,r} (\tilde{I})\ra \subseteq I_n.
\]
For the reverse inclusion we use induction on $n \ge r$. This is clear if $n = r$.  If $n > r$ we get
\[
I_n = \la \Inc^i_{n-1,n} (I_{n-1}) \ra = \la \Inc^i_{n-1,n} (J_{n-1}) \ra \subseteq J_n.
\]
We conclude that $J_n = I_n$ if $n \ge r$.
Finally, using that each $\pi \in \Inc^i$ is in $\Inc^i_{r, \pi(r)}$, it is easy to see that $J = \bigcup_{n \in \N} J_n$ is the $\Inc^i$-closure of $\tilde{I}$.
\end{proof}

\begin{rem}
Note that, with the notation of Lemma \ref{lem:equivfilt}, one has
\[
J \cap K[X_n] = I_n = J_n \quad \text{ for all } n \ge r.
\]
Indeed, this can be shown similarly as \cite[Proposition 2.10]{HS} using \cite[Lemma 2.18]{HS}.

\end{rem}


\section{Hilbert series up to symmetry}

 Recall that the equivariant or bigraded Hilbert series of a filtration of ideals $\Ic = (I_n)_{n \in \N}$ is defined as
\[
H_{\Ic} (s, t) = \sum_{n \ge 0,\, j \ge 0} \dim_K [K[X_n]/I_n]_j \cdot s^n t^j,
\]
where $I_0 = 0$, and thus $K[X_0]/I_0 \cong K$. For the remainder of this section we restrict ourselves to considering  filtrations of monomial ideals.

\begin{ex}
   \label{exa:trivial filtration}
If a filtration $\Ic = (I_n)_{n \in \N}$ is zero, that is, $I_n = 0$ for all $n$, then its bigraded Hilbert series
\[
H_{\Ic} (s, t) = \sum_{n \ge 0} \frac{1}{(1-t)^{cn}} s^n = \sum_{n \ge 0} \left(\frac{s}{(1-t)^{c}} \right)^n= \frac{(1-t)^c}{(1-t)^c - s}
\]
is rational. In the other trivial case where $I_n = K[X_n]$ for all $n \in \N$, one has $H_{\Ic} (s, t) = 1$.
\end{ex}

 The goal of this section is to show that  the situation in the example is typical. We denote by $G(J)$ the minimal system of monomial generators of a monomial ideal $J$ which is uniquely determined. Furthermore, we write  $e^+ (J)$  for  the maximum degree of a minimal generator of a graded ideal $J$.

\begin{thm}
    \label{thm:rat Hilb series, rev}
Let $ i \ge 0$ and consider any  $\Inc^i$-invariant  filtration of monomial ideals  $\Ic = (I_n)_{n \in \N}$.  Set $r  = ind^i (\Ic)$ and $q = \sum_{j = 0}^{e^+ (I_r)}  \dim_K [K[X_r]/I_r]_j$.
Then the bigraded Hilbert series $H_{\Ic} (s, t)$  of $\Ic$ is a  rational function in $s$ and $t$  of the form
\[
H_{\Ic} (s, t) = \frac{g(s, t)}{(1-t)^a \cdot \prod_{j =1}^b [(1-t)^{c_j} - s \cdot f_j (t)]},
\]
where $a, b, c_j \in \N_0$ with $c_j \le c$, \ $g (s, t) \in \Z[s, t]$, each $f_j (t) \in \Z[t]$, and $f_j (1) >  0$.

Furthermore,
$a \le (r-1+ 2q) c$,   the degree of   $g(s,1)$  is at most  $r+q$, and \ $b \le  \frac{(d+1)^{q c} - 1}{(d+1)^c - 1}$,
where
\[
d =  \max\{ e \in \N_0 \; | \; x_{k, i+1}^e \text{ divides some monomial in } G(I_r) \text{ for some } k \in [c]\}.
\]
\end{thm}

For the proof, we begin with a special case.

\begin{lem}
    \label{lem: p is zero}
Consider any integer $i \ge 0$ and any non-trivial $\Inc^i$-invariant filtration $\Ic = (I_n)_{n \in \N}$ such that $r = \ind^i (\Ic) \le i$. Then  its bigraded Hilbert series $H_{\Ic} (s, t)$  is
\[
H_{\Ic} (s, t) = \frac{g (s, t)}{(1-t)^a \cdot [(1-t)^c - s]},
\]
where   $a =  \max\{\dim K[X_n]/I_n \; |  \; 1 \le n < r\}$,  $g(s,t)  \in \Z[s,t]$, and  $g(s,1) \neq 0$ is a polynomial in $s$  whose coefficients are all non-positive and whose degree is at most $r$.

Moreover, 
if $I_r = K[X_r]$, then $g(s, t)$ is divisible by $(1-t)^c - s$.
\end{lem}

\begin{proof}
Since $r \le i$ by assumption, the action of $\Inc^i$ leaves each polynomial in $K[X_r]$ invariant. It follows that, for each $n \ge r$, one has $K[X_n]/I_n \cong (K[X_r]/I_r) [X_n \setminus X_r]$. Using that $|X_n \setminus X_r| = c (n-r)$, this  implies
\[
H_{K[X_n]/I_n} (t) = H_{K[X_r]/I_r}(t)  \cdot \frac{1}{(1-t)^{c (n-r)}}.
\]
Hence, we obtain
\begin{align*}
H_{\Ic} (s,t) & = \sum_{n=0}^{r-1} H_{K[X_n]/I_n} (t) \cdot s^n + \sum_{n\ge r}  H_{K[X_n]/I_n} (t) \cdot s^n \\
& = \sum_{n=0}^{r-1} H_{K[X_n]/I_n} (t) \cdot s^n + \sum_{n\ge r}  H_{K[X_r]/I_r} (t) \cdot \frac{s^n}{(1-t)^{c (n-r)}} \\
& = \sum_{n=0}^{r-1} H_{K[X_n]/I_n} (t) \cdot s^n + (1-t)^{c r} \cdot  H_{K[X_r]/I_r} (t) \cdot \sum_{n\ge r}  \left (\frac{s}{(1-t)^c} \right )^n\\
& = \sum_{n=0}^{r-1} H_{K[X_n]/I_n} (t) \cdot s^n + (1-t)^{c r} \cdot  H_{K[X_r]/I_r} (t) \cdot \left (\frac{s}{(1-t)^c} \right )^r \cdot \frac{(1-t)^c}{(1-t)^c - s} \\
& = \sum_{n=0}^{r-1} \frac{g_n (t)}{(1-t)^{d_n}} \cdot s^n +    \frac{ g_r (t) s^r}{(1-t)^{d_r - c} [(1-t)^c - s]},
\end{align*}
where $H_{K[X_n]/I_n} (t) = {g_n (t)}/ {(1-t)^{d_n}}$ with $d_n = \dim K[X_n]/I_n$. This is also true if $g_r (t) = 0$. It follows that
\[
H_{\Ic} (s,t)  = \frac{1}{(1-t)^{a} \cdot  [(1-t)^c - s]} \sum_{j = 0}^r p_j (t) s^j
\]
with polynomials $p_0 (t) = (1-t)^{a+c - d_0} g_0 (t)$ and, for $j \in [r]$,
\[
p_j (t) = (1-t)^{a - d_{j-1}} \left [(1-t)^{c+d_{j-1} - d_j} g_j (t) - g_{j-1} (t)\right  ].
\]
Note that $d_j \le d_{j-1} + c$. Hence, we get $p_j (1) = 0$ if $j = 0$ or $d_j < a$ and otherwise
\[
p_j (1) = \begin{cases}
- g_{j-1} (1) = - \deg I_{j-1} \le  0 & \text{if } d_j < d_{j-1} + c, \\
\deg I_j - \deg I_{j-1} \le  0 & \text{if } d_j = d_{j-1} + c.
\end{cases}
\]
Here the last estimate is true because   $d_j = d_{j-1} + c$ gives $\dim K[X_j]/I_j = \dim K[X_j]/I_{j-1} K[X_j]$, and thus we conclude by comparing the leading coefficients of Hilbert polynomials using the fact $I_{j-1} K[X_j] \subset I_j$. It follows that  $H_{\Ic} (s, t)$ has  the desired form. Furthermore,  if $I_r = K[X_r]$ then $g_r (t) = 0$, which implies the statement in the special case.
\end{proof}

In order to establish Theorem \ref{thm:rat Hilb series, rev} we need several further preliminary results. The first observation
describes a certain invariance when manipulating a filtration.

\begin{lem}
    \label{lem:invariance under fixed i}
Let $\Ic = (I_n)_{n \in \N}$ be an $\Inc^i$-invariant filtration of monomial ideals with $i < \ind^i (\Ic)$. For any variable $x_{k, i} \in X_i \setminus X_{i-1}$ and any integer $e > 0$, consider filtrations $(\Ic, x_{k,i})$ and $(\Ic : x_{k,i}^e)$ whose $n$-th ideals are $\la I_n, x_{k,i} \ra$ and $I_n : x_{k,i}^e$, respectively, if $n \ge i$ and zero if $n < i$. Then both filtrations are $\Inc^i$-invariant and
\[
\ind^i (\Ic, x_{k, i}) \le \ind^i (\Ic) \quad \text{ and } \quad \ind^i (\Ic : x_{k, i}^e) \le \ind^i (\Ic).
\]
\end{lem}

\begin{proof}
To see the second assertion  consider any monomial $u \in I_n : x_{k, i}^e$, where $n > \ind^i (\Ic)$.  Then, for each integer $m$ with $\ind^i (\Ic) \le m \le n$, there is some monomial $v \in I_m$ and some    $\pi \in \Inc^i_{m, n}$ such that
\[
 r \pi (v) = x_{k, i}^e u \in I_n
\]
for some monomial $r \in K[X_n]$. Write $v = x_{k, i}^{\tilde{e}} \tilde{v}$ for some integer $\tilde{e}$ with $0 \le \tilde{e} \le e$ and some $\tilde{v} \in K[X_m]$, where $x_{k, i}$ does not divide $\tilde{v}$ if $\tilde{e} < e$. Observe that $\tilde{v} \in I_m : x_{k, i}^{e}$. Since $\pi \in \Inc^i$ fixes $x_{k, i}$ we get
\[
x_{k, i}^e u  =  r \pi (v) =  r x_{k, i}^{\tilde{e}} \pi(\tilde{v}),
\]
where $x_{k, i}$ does not divide $\pi (\tilde{v})$ if $\tilde{e} < e$. In any case, $x_{k, i}^e$ divides $ r x_{k, i}^{\tilde{e}}$. It follows that $u$ is a monomial multiple of $\pi(\tilde{v})$, and thus $u \in \la \Inc^i_{m,n} (I_m : x_{k, i}^e)\ra$. Similarly, one sees that $(\Ic : x_{k,i}^e)$ is in fact an $\Inc^i$-invariant filtration, which completes the proof for $(\Ic : x_{k,i}^e)$.

Analogous, but easier arguments give the assertions for $(\Ic, x_{k,i})$.
\end{proof}

We now observe that restricting the acting monoid preserves filtrations.

\begin{lem}
   \label{lem:compare two filtrations}
Consider, for some $i \ge 0$, an $\Inc^i$-invariant filtration $(I_n)_{n \in \N}$ of monomial ideals. If $n > m \ge 1$, then, as ideals of $K[X_n]$,
\[
\la \Inc^i_{m, n} (I_m) \ra \subset \la \Inc^{i+1}_{m+1, n}  (I_{m+1}) \ra \subset \la \Inc^{i}_{m+1. n} (I_{m+1})\ra.
\]
\end{lem}

\begin{proof}
The second inclusion is clearly true. The first inclusion follows from  Proposition \ref{prop:inc-decomp}.
\end{proof}

We draw two useful consequences.

\begin{cor}
    \label{cor:index change when increasing i}
If $\Ic = (I_n)_{n \in \N}$  is an  $\Inc^i$-invariant filtration of monomial ideals, then it is also an $\Inc^{i+1}$-invariant filtration and
\[
\ind^{i+1} (\Ic) \le 1 + \ind^i (\Ic).
\]
\end{cor}

\begin{proof}
If $n > m \ge \ind^i (\Ic)$, then Lemma \ref{lem:compare two filtrations} gives
\[
I_n = \la \Inc^i_{m, n} (I_m) \ra \subset \la \Inc^{i+1}_{m+1, n} (I_{m+1}) \ra \subset  \la \Inc^{i}_{m+1. n} (I_{m+1}) \ra \subset I_n,
\]
which forces equality everywhere.
\end{proof}

\begin{cor}
    \label{cor:index estimate}
Consider, for some $i \ge 0$, an $\Inc^i$-invariant filtration $(I_n)_{n \in \N}$ of monomial ideals and any variable $x_{k, i} \in X_{i} \setminus X_{i-1}$.
Then    $(\Ic, x_{k, i})$ is an $\Inc^{i+1}$-invariant filtration and
\[
\ind^{i+1} (\Ic, x_{k, i}) \le 1 + \ind^i (\Ic).
\]
\end{cor}

\begin{proof}
Combine Lemma \ref{lem:invariance under fixed i} and Corollary \ref{cor:index change when increasing i}.
\end{proof}

The following result  is not restricted to monomial ideals. It is elementary, but very useful for our purpose.

\begin{lem}
   \label{lem:one table}
Let $I \subset R = K[X_n]$ be a graded ideal, and let $\ell \in R$ be a linear form such that $I : \ell^d = I : \ell^{d+1}$ for some positive integer $d$. Then
\[
H_{R/I} (t) = \sum_{e = 0}^{d-1} H_{R/\langle I: \ell^e, \ell\rangle } (t) \cdot t^e
+ H_{R/\langle I : \ell^d, \ell\rangle} (t) \cdot \frac{t^d}{1-t} .
\]
\end{lem}

\begin{proof}
Consider the following exact sequences, induced by multiplication by $\ell$:
\[
\begin{array}{lllllllll}
0 &\to &R/I : \ell(-1) &\overset{\cdot \ell}{\to} &R/I &\to &R/\langle I, \ell\rangle  &\to &0
\\
0 &\to &R/I : \ell^2(-1) &\overset{\cdot \ell}{\to} &R/I : \ell &\to &R/\langle I: \ell, \ell\rangle &\to &0
\\
\vdots&&&&\vdots&&&&\vdots
\\
0 &\to &R/I : \ell^d(-1) &\overset{\cdot \ell}{\to} &R/I : \ell^{d-1} &\to &R/\langle I: \ell^{d-1}, \ell\rangle &\to &0
\\
0 &\to &R/I : \ell^{d+1}(-1) &\overset{\cdot \ell}{\to} &R/I : \ell^d &\to &R/\langle I: \ell^{d}, \ell\rangle &\to &0
\end{array}.
\]
Using all exact sequences but the last one,  we obtain
\[
H_{R/I} (t) = \sum_{e = 0}^{d-1} H_{R/\langle I: \ell^e, \ell\rangle} (t) \cdot t^e + H_{R/ I : \ell^d} (t) \cdot t^d.
\]
The assumption $I : \ell^d = I : \ell^{d+1}$ and the last sequence give
\[
H_{R/I : \ell^d} (t)\cdot(1-t) =  H_{R/\langle I : \ell^d, \ell\rangle}.
\]
Combining the two equations our claim follows.
\end{proof}
Applying this idea to a filtration, we obtain:

\begin{cor}
    \label{cor: one table, bigraded}
Consider, for some $i\ge 0$,  an $\Inc^{i+1}$-invariant filtration $\Ic = (I_n)_{n \in \N}$ of graded ideals. Assume there are  integers $r \ge i+1$, \,  $d > 0$  and a linear form $\ell \in K[X_{i+1}]$ such that
\[
I_n : \ell^d = I_n : \ell^{d+1} \text{ for all } n \ge r.
\]
For each $e \in \{0,\ldots,d\}$, define a sequence of ideals $\langle \Ic : \ell^e, \ell\rangle
= (J_{n,e})_{n \in \N}$ by
\[
J_{n,e} = \begin{cases}
\langle 0 \rangle & \text{if } n < r, \\
\langle I_n : \ell^e, \ell\rangle &  \text{if } n \ge r.
\end{cases}
\]
Then each $\langle \Ic : \ell^e, \ell\rangle$ is an $\Inc^{i+1}$-invariant filtration,
and there is some polynomial $g (s, t)\in \Z[s, t]$  with $g(s, 1) = - s^{r-1}$ such that
\[
H_{\Ic} (s, t) = \sum_{e = 0}^{d-1} H_{\langle \Ic : \ell^e, \ell\rangle} (s, t) \cdot t^e + H_{\langle \Ic : \ell^d, \ell\rangle} (s, t) \cdot \frac{t^d}{1-t} + \frac{g(s, t)}{(1-t)^{(r-1) c+1}}.
\]
\end{cor}

\begin{proof}
Notice that each sequence $\langle \Ic : \ell^e, \ell\rangle$ is  an $\Inc^{i+1}$-invariant filtration by observing $\ell  \pi (f) = \pi  ( \ell f)$ for each $\pi \in \Inc^{i+1}$ and each polynomial $f$ because $\ell \in K[X_{i+1}]$.  By assumption, Lemma \ref{lem:one table} applies to each ideal $I_n$ with $n \ge r$.  This provides
\begin{eqnarray*}
H_{\Ic} (s, t)
&=&
\sum_{e = 0}^{d-1} H_{\langle \Ic : \ell^e, \ell\rangle} (s, t) \cdot t^e + H_{\langle \Ic : \ell^d, \ell\rangle} (s, t) \cdot \frac{t^d}{1-t}
\\
&&
-
\sum_{n=0}^{r-1} \sum_{e = 0}^{d-1} H_{K[X_n]/J_{n,e}}(t) \cdot s^n\cdot t^e
-
\sum_{n=0}^{r-1} H_{K[X_n]/J_{n,d}}(t) \cdot s^n\cdot \frac{t^d}{1-t}
+
\sum_{n=0}^{r-1} H_{K[X_n]/I_n} (t) \cdot s^n
.
\end{eqnarray*}
Since for $1\le n \le r-1$ the Krull dimension of $K[X_n]/I_n$
and of $K[X_n]/J_{n,e}=K[X_n]$ is at most $(r-1) c$, we can write
\begin{eqnarray*}
&-&
\sum_{n=0}^{r-1} \sum_{e = 0}^{d-1} H_{K[X_n]/J_{n,e}}(t) \cdot s^n\cdot t^e
-
\sum_{n=0}^{r-1} H_{K[X_n]/J_{n,d}}(t) \cdot s^n\cdot \frac{t^d}{1-t}
+
\sum_{n=0}^{r-1} H_{K[X_n]/I_n} (t) \cdot s^n
\\&= &
-
\sum_{n=0}^{r-1} \sum_{e = 0}^{d-1} \frac{s^n\cdot t^e}{(1-t)^{cn}}
-
\sum_{n=0}^{r-1} \frac{s^n}{(1-t)^{c n}}\cdot \frac{t^d}{1-t}
+
\sum_{n=0}^{r-1} H_{K[X_n]/I_n} (t) \cdot s^n
\\ &=
&
\frac{g(s, t)}{(1-t)^{(r-1) c+1}}
\end{eqnarray*}
for some $g (s, t)\in \Z[s, t]$ with $g(s, 1) = - s^{r-1}$. Now our assertion follows.
\end{proof}

We now iterate the construction in the previous result.

\begin{lem}
    \label{lem:multiple tables}
Consider, for some $i \ge 0$,  an $\Inc^i$-invariant filtration $\Ic = (I_n)_{n \in \N}$ of monomial ideals.  Fix an integer $r > \ind^i (\Ic) \ge i$. Let $d > 0$ be an integer such that, for each $k \in [c]$, the monomial $x_{k, i+1}^{d+1}$ does not divide any monomial in $G(I_r)$.  For non-negative integers $e_1,\ldots,e_c$, consider a  sequence
\[
\langle \Ic : x_{1,i+1}^{e_1} \cdots x_{c,i+1}^{e_c}, x_{1,i+1},\ldots,x_{c,i+1}\rangle
\]
whose $n$-th ideal is $\langle I_n  : x_{1,i+1}^{e_1} \cdots x_{c,i+1}^{e_c}, x_{1,i+1},\ldots,x_{c,i+1}\rangle $ if $n \ge r$ and $\langle 0\rangle$ if $1 \le n < r$. Then $\langle \Ic : x_{1,i+1}^{e_1} \cdots x_{c,i+1}^{e_c}, x_{1,i+1},\ldots,x_{c,i+1}\rangle$ is an $\Inc^{i+1}$-invariant filtration whose $(i+1)$-index equals $r$, and
\[
H_{\Ic} (s, t) = \frac{1}{(1-t)^c} \cdot
\sum_{\be = (e_1,\ldots,e_c)\in \Z^c \atop 0 \le e_l \le d}
f_{\be} (t) \cdot H_{\langle \Ic : x_{1,i+1}^{e_1} \cdots x_{c,i+1}^{e_c}, x_{1,i+1},\ldots,x_{c,i+1}\rangle} (s, t) + \frac{g(s, t)}{(1-t)^{rc}},
\]
where
\[
f_{\be} (t)  =  t^{|\be|} (1-t)^{\delta (\be)} \; \text{ with} \ |\be| = e_1 + \cdots + e_c, \   \delta(\be) = \# \{e_l  \; | \;  e_l = d \text{ and } 1 \le l \le c\},
\]
and
\[
g (s, t)\in \Z[s, t] \text{ with } g(s, 1) = - d^{c-1} s ^{r-1}.
\]
\end{lem}

\begin{proof}
According to Corollary \ref{cor:index change when increasing i}, the sequence $\Ic$ is an $\Inc^{i+1}$-invariant filtration with $\ind^{i+1} (\Ic) \le r$. Thus, $\langle \Ic : x_{1,i+1}^{e_1} \cdots x_{c,i+1}^{e_c}, x_{1,i+1},\ldots,x_{c,i+1}\rangle$ is an $\Inc^{i+1}$-invariant filtration whose $(i+1)$-index is at most $r$ by repeatedly using  Lemma \ref{lem:invariance under fixed i}. Hence the $(i+1)$-index  equals $r$ by definition of the sequence.

Note that  no monomial $x_{k,i+1}^{d+1}$ divides any monomial in $G(I_n)$ if $n \ge r$
because of the choice of $d$ and $\ind^{i+1} (\Ic) \le r$.
Hence,
\begin{align*}
\hspace{2em}&\hspace{-2em}
\langle I_n : x_{1,i+1}^{e_1} \cdots x_{k-1,i+1}^{e_{k-1}} x_{k,i+1}^{d}, \ x_{1,i+1},\ldots,x_{k-1,i+1}\rangle  \\
& \hspace{2em}
= \langle
I_n : x_{1,i+1}^{e_1} \cdots x_{k-1,i+1}^{e_{k-1}} x_{k,i+1}^{d+1}, \ x_{1,i+1},\ldots,x_{k-1,i+1}\rangle
\end{align*}
whenever $0 \le e_j \le d$ and $1 \le k \leq c$.

In the following we will use several times the fact that
\begin{align*}
\hspace{2em}&\hspace{-2em}
\langle
I_n : x_{1,i+1}^{e_1} \cdots x_{k-1,i+1}^{e_{k-1}} x_{k,i+1}^{e_k}, x_{1,i+1},\ldots,x_{k-1,i+1}\rangle
 \\
& \hspace{2em}
= \langle
I_n : x_{1,i+1}^{e_1} \cdots x_{k-1,i+1}^{e_{k-1}}, x_{1,i+1},\ldots,x_{k-1,i+1}
\rangle : x_{k,i+1}^{e_k}.
\end{align*}
In particular, for $e_k=d$, this guarantees that the assumption
on the ideal quotients in Corollary \ref{cor: one table, bigraded} is satisfied. Consider  filtrations
$\langle
\Ic: x_{1,i+1}^{e_1} \cdots x_{k,i+1}^{e_{k}}, x_{1,i+1},\ldots,x_{k,i+1}
\rangle$
with $1 \le k  \le c$ by defining its $n$-th ideal as
\[
\langle
I_n: x_{1,i+1}^{e_1} \cdots x_{k,i+1}^{e_{k}}, x_{1,i+1},\ldots,x_{k,i+1}
\rangle \text{ if } n \ge r \text{  and as  } \langle 0\rangle \text{  if } 1 \le n < r.
\]
Now we prove by induction on $k$, where $1 \le k \le c$,   that
\begin{align*}
\hspace{2em}&\hspace{-2em} H_{\Ic} (s, t)  = \\
& \frac{1}{(1-t)^{k}}
\cdot
\sum_{\be = (e_1,\ldots,e_k)\in \Z^k \atop 0 \le e_l \le d}
f_{\be,k} (t)
\cdot
H_{\langle \Ic : x_{1,i+1}^{e_1} \cdots x_{k,i+1}^{e_k}, x_{1,i+1},\ldots,x_{k,i+1}\rangle} (s, t)
+
\frac{g_k(s, t)}{(1-t)^{(r-1) c+k}},
\end{align*}
where
\[
f_{\be,k}  (t)  = t^{|\be|_k} (1-t)^{\delta_k (\be)} \; \text{ with} \ |\be|_k = e_1 + \cdots + e_k,  \   \delta_k (\be) = \# \{e_l  \; | \;  e_l = d \text{ and } 1 \le l \le k\}\]
and $g_k(s, t)\in \Z[s, t]$ with $g_k (s, 1) = - d^{k-1} s^{r-1}$.

Let $k=1$. We observed above that if $n \ge r$ then
\[
I_n : x_{1,i+1}^{d}
=
 I_n : x_{1,i+1}^{d+1}.
\]
Corollary \ref{cor: one table, bigraded}
implies that there exists $g_1(s, t)\in \Z[s,t]$ with $g_1(s, 1) = - s^{r-1}$ such that
\begin{align*}
H_{\Ic} (s, t)  & = \sum_{e = 0}^{d-1} H_{\langle \Ic : x_{1,i+1}^e, x_{1,i+1}\rangle} (s, t) \cdot t^e + H_{\langle \Ic : x_{1,i+1}^d, x_{1,i+1}\rangle} (s, t) \cdot \frac{t^d}{1-t} + \frac{g_1(s, t)}{(1-t)^{(r-1) c+1}}
\\
& =  \frac{1}{(1-t)}
\cdot
\sum_{e = 0}^{d} f_{e,1} (t)
\cdot
H_{\langle \Ic : x_{1,i+1}^e, x_{1,i+1}\rangle} (s, t) + \frac{g_1(s, t)}{(1-t)^{(r-1) c+1}},
\end{align*}
where
$f_{e,1} (t)=t^e(1-t)
\text{ if } 0\le e \le d-1
\text{ and }
f_{d,1} (t)=t^d$, as claimed.

Next let $k>1$.  We observed at the beginning of the proof that
\begin{align*}
\hspace{2em}&\hspace{-2em} \langle
I_n : x_{1,i+1}^{e_1} \cdots x_{k-1,i+1}^{e_{k-1}}, x_{1,i+1},\ldots,x_{k-1,i+1}
\rangle : x_{k,i+1}^{d}  \\
& \hspace{2em} = \langle
I_n : x_{1,i+1}^{e_1} \cdots x_{k-1,i+1}^{e_{k-1}}, x_{1,i+1},\ldots,x_{k-1,i+1}
\rangle : x_{k,i+1}^{d+1}.
\end{align*}
Corollary \ref{cor: one table, bigraded} applied to $\langle \Ic : x_{1,i+1}^{e_1} \cdots x_{k-1,i+1}^{e_{k-1}}, x_{1,i+1},\ldots,x_{k-1,i+1}\rangle$
implies that there exists $\Tilde{g}_k(s, t)\in \Z[s,t]$ with $\Tilde{g}_{\be, k-1} (s, 1) = - s^{r-1}$ such that
\begin{eqnarray*}
&&H_{\langle \Ic : x_{1,i+1}^{e_1} \cdots x_{k-1,i+1}^{e_{k-1}}, x_{1,i+1},\ldots,x_{k-1,i+1}\rangle} (s, t)
\\
&=&
\sum_{e_k = 0}^{d-1}
H_{\langle \Ic : x_{1,i+1}^{e_1} \cdots x_{k,i+1}^{e_k}
, x_{1,i+1},\ldots, x_{k,i+1}\rangle} (s, t) \cdot t^{e_k}
+
H_{\langle \Ic : x_{1,i+1}^{e_1} \cdots x_{k,i+1}^{d}
, x_{1,i+1},\ldots, x_{k,i+1}\rangle} (s, t) \cdot \frac{t^d}{1-t}
\\
&&+ \frac{\Tilde{g}_{\be, k-1} (s, t)}{(1-t)^{(r-1) c+1}}
\\
&=&
\frac{1}{(1-t)}
\cdot
\sum_{e_k = 0}^{d}
\Tilde{f}_{e_k} (t)
\cdot
H_{\langle \Ic : x_{1,i+1}^{e_1} \cdots x_{k,i+1}^{e_k}
, x_{1,i+1},\ldots,x_{k,i+1}\rangle} (s, t)
+
\frac{\Tilde{g}_{\be, k-1} (s, t)}{(1-t)^{(r-1) c+1}},
\end{eqnarray*}
where
 $\Tilde{f}_{e_k} (t)=t^{e_k}(1-t)
\text{ for } 0\le e_k \le d-1
\text{ and }
\Tilde{f}_{d} (t)=t^d$.

The induction hypothesis  yields
\begin{eqnarray*}
H_{\Ic} (s, t)
&=&
\frac{1}{(1-t)^{k-1}}
\cdot
\sum_{\be = (e_1,\ldots,e_{k-1})\in \Z^{k-1} \atop 0 \le e_l \le d}
f_{\be,k-1} (t)
\cdot
H_{\langle \Ic : x_{1,i+1}^{e_1} \cdots x_{k-1,i+1}^{e_{k-1}}, x_{1,i+1},\ldots,x_{k-1,i+1}\rangle} (s, t)
\\
&&+
\frac{g_{k-1}(s, t)}{(1-t)^{(r-1) c+k-1}},
\end{eqnarray*}
where
\[
f_{\be,k-1}  (t)  = t^{|\be |_{k-1}} (1-t)^{\delta_{k-1} (\be)}
\]
and $g_{k-1}(s, t)\in \Z[s, t]$ with $g_{k-1} (s, 1) = - d^{k-2} s^{r-1}$.

 Substituting the formula for  $H_{\langle \Ic : x_{1,i+1}^{e_1} \cdots x_{k-1,i+1}^{e_{k-1}}, x_{1,i+1},\ldots,x_{k-1,i+1}\rangle} (s, t)$ into the last equation and simplifying gives
\begin{align*}
H_{\Ic} (s, t)
=&
\frac{1}{(1-t)^{k}}
\cdot
\sum_{\be = (e_1,\ldots,e_{k-1})\in \Z^{k-1} \atop 0 \le e_l \le d}
\bigl(
\sum_{e_k = 0}^{d}
f_{\be,k-1} (t)\Tilde{f}_{e_k} (t)
\cdot
H_{\langle \Ic : x_{1,i+1}^{e_1} \cdots x_{k,i+1}^{e_k}
, x_{1,i+1},\ldots x_{k,i+1}\rangle} (s, t)
\bigr)
\\
&
\ +
\frac{1}{(1-t)^{(r-1) c + k}}
\cdot
\sum_{\be = (e_1,\ldots,e_{k-1})\in \Z^{k-1} \atop 0 \le e_l \le d}
f_{\be,k-1} (t)   \Tilde{g}_{\be, k-1} (s, t)
+
\frac{g_{k-1}(s, t)}{(1-t)^{(r-1) c+k-1}},
\\
\\
=&
\frac{1}{(1-t)^{k}}
\cdot
\sum_{\be = (e_1,\ldots,e_{k})\in \Z^{k} \atop 0 \le e_l \le d}
f_{\be,k} (t)
\cdot
H_{\langle \Ic : x_{1,i+1}^{e_1} \cdots x_{k,i+1}^{e_k}
, x_{1,i+1},\ldots,x_{k,i+1}\rangle} (s, t)
\bigr)
+
\frac{g_{k}(s, t)}{(1-t)^{(r-1) c+k}},
\end{align*}
where
\begin{align*}
f_{(e_1,\dots,e_k),k} (t) & =  f_{(e_1,\dots,e_{k-1}),k-1} (t)\cdot\Tilde{f}_{e_k} (t)
 =   t^{|\be|_k} (1-t)^{\delta_k (\be)}
\end{align*}
and
\[
g_{k}(s, t)
=
 \sum_{\be = (e_1,\ldots,e_{k-1})\in \Z^{k-1} \atop 0 \le e_l \le d}
f_{\be,k-1} (t)  \Tilde{g}_{\be, k-1} (s, t)
 +
g_{k-1}(s, t)\cdot (1-t) \in \Z[s,t].
\]
Observe that
\begin{align*}
g_k(s, 1) & = \sum_{\be = (e_1,\ldots,e_{k-1})\in \Z^{k-1} \atop 0 \le e_l \le d}
f_{\be,k-1} (1)  \Tilde{g}_{\be, k-1} (s, 1)  \\
& = \sum_{\be = (e_1,\ldots,e_{k-1})\in \Z^{k-1} \atop 0 \le e_l \le d, \ \delta_{k-1} (\be)= 0}
 (-  s^{r-1}) \\
& = - d^{k-1}s^{r-1}.
\end{align*}
The case $k=c$ of the induction establishes the Lemma.
\end{proof}

The following result uses the elements $\sigma_i$ defined at the beginning of Section \ref{secLdecomp}.

\begin{lem}
    \label{lem:shifted filtrations}
Consider, for some $i \ge 0$,  an $\Inc^i$-invariant filtration $\Ic = (I_n)_{n \in \N}$ of monomial ideals.
Assume  $\ind^i (\Ic) \ge i+1$. Let  $r \ge \ind^i (\Ic)$ be an integer,
and set
\[
J_{r+1} = \la \sigma_i (I_r), \ x_{1,i+1},\ldots,x_{c,i+1}\ra.
\]
Define a sequence of ideals $\Jc = (J_n)_{n \in \N}$ by
\[
J_n = \begin{cases}
0 & \text{if } 1 \le n \le r, \\
\la \Inc_{r+1, n}^{i+1} (J_{r+1}) \ra & \text{if } n \ge r+1.
\end{cases}
\]
Then $\Jc$ is an $\Inc^{i+1}$-invariant filtration, and there exists a polynomial  $g (s, t)\in \Z[s, t]$ with $g(s,1) = s^r$ such that
\[
H_{\mathcal J } (s, t) =  s \cdot H_{\Ic} (s, t) + \frac{g(s, t)}{(1-t)^{rc}}.
\]
\end{lem}

\begin{proof}
By Lemma \ref{lem:equivfilt}
we know that $\Jc$ is an $\Inc^{i+1}$-invariant filtration
with stability index
$$
\ind^{i+1} (\mathcal J)=r+1.
$$
At first we prove by induction on $n \geq r$ that
\[
J_{n+1}=
\la \sigma_i (I_{n}), \ x_{1,i+1},\ldots,x_{c,i+1}\ra.
\]
If  $n=r$ , then this is true by one of our assumptions.

Let $n>r$. We have
\begin{eqnarray*}
\la \sigma_i (I_{n}), \ x_{1,i+1},\ldots,x_{c,i+1}\ra
&=&
\la
\sigma_i \bigl(\la \Inc^{i}_{r,n} (I_{r} ) \ra\bigr), \ x_{1,i+1},\ldots,x_{c,i+1}\ra
\\
&\subseteq&
\la
\Inc^{i+1}_{r+1,n+1} \bigl(\sigma_i  (I_{r}) \bigr), \ x_{1,i+1},\ldots,x_{c,i+1}\ra
\\
&=&
\la \Inc_{r+1, n+1}^{i+1} (J_{r+1}) \ra
\\
&=&
J_{n+1},
\end{eqnarray*}
where the containment is a consequence of Lemma \ref{lem:element comp}.
It remains to show the reverse inclusion. Since $n\geq r+1=\ind^i (\mathcal J)$
the induction hypothesis implies
\begin{eqnarray*}
J_{n+1}
&=&
\la \Inc^{i+1}_{n,n+1} (J_n) \ra
\\
&=&
\la
\Inc^{i+1}_{n,n+1}
(  \la \sigma_i (I_{n-1}), \ x_{1,i+1},\ldots,x_{c,i+1}\ra)
\ra
\\
&=&
\la
\Inc^{i+1}_{n,n+1} \bigl(\sigma_i (I_{n-1})\bigr)
, \ x_{1,i+1},\ldots,x_{c,i+1}
\ra.
\end{eqnarray*}
Note that $\la
\Inc^{i+1}_{n,n+1}  \bigl(\sigma_i (I_{n-1}) \bigr) \ra = \la \{\sigma_{i+1},\ldots,\sigma_{n+1} \}  \bigl(\sigma_i (I_{n-1}) \bigr) \ra$ because for any $f \in K[X_n]$ and any $\pi \in \Inc^i_{n, n+1}$, there is some $\sigma_j$ with $i < j \le n+1$ such that $\pi (f) = \sigma_j (f)$.
Using Corollary \ref{cor:sigma comp}, it follows that
\begin{eqnarray*}
J_{n+1}
&\subseteq&
\la
\sigma_i \bigl(\Inc^{i}_{n-1,n}\ (I_{n-1}) \bigr)
, \ x_{1,i+1},\ldots,x_{c,i+1}
\ra
\\
&\subseteq&
\la
\sigma_i (I_n)
, \ x_{1,i+1},\ldots,x_{c,i+1}\ra.
\end{eqnarray*}
Hence
\[
J_{n+1}=
\la
\sigma_i (I_n)
, \ x_{1,i+1},\ldots,x_{c,i+1}\ra,
\]
concluding the proof by induction.

Next we observe that $\sigma_i (I_n)$
is obtained from $I_n$ replacing  $x_{k,l}$ by $x_{k,l+1}$ for $1\leq k\leq c$ and $i+1\leq l$. In particular, $\sigma_i (I_n)$ has no minimal monomial generator which is divisible by any $x_{k,i+1}$ for $1\leq k\leq c$.
Hence  for $n \ge r$, the map
\[
K[X_n]/I_n \to K[X_{n+1}]/\la \sigma_i (I_n), \ x_{1,i+1},\ldots,x_{c,i+1}\ra, \;
\overline{x_{k,l}} \mapsto
\begin{cases}
\overline{x_{k,l}}& \text{if } l\le i,\\
\overline{x_{k,l+1}}& \text{if } i+1\leq l\leq n,
\end{cases}
\]
is well-defined and an isomorphism of graded $K$-algebras.
This implies
\[
H_{K[X_n]/I_n}(t)=H_{K[X_{n+1}]/J_{n+1}}(t)
\text{ for } n\geq r.
\]
Thus
\begin{eqnarray*}
H_{\mathcal J } (s, t)
&=&
s \cdot H_{\Ic} (s, t) + \sum_{n=0}^{r} H_{K[X_{n}]}(t)s^n - s \sum_{n=0}^{r-1} H_{K[X_n]/I_n} (t) s^n
\\
&=&
s \cdot H_{\Ic} (s, t) + \frac{g(s, t)}{(1-t)^{rc}},
\end{eqnarray*}
where $g(s, t) \in \Z[s,t]$ and $g(s,1) = s^r$.
Here we used the fact that  $H_{K[X_n]} (t) = \frac{1}{(1-t)^{c n}}$ and   that the Krull dimension of
$K[X_n]/I_n$ is at most $(r-1)c$ if $0\le n \le r-1$.
This concludes the proof.
\end{proof}

We are ready to establish the main result of this section.

\begin{proof}[Proof of Theorem \ref{thm:rat Hilb series, rev}]
Our argument uses a double induction.    We show by a first  induction on $p \ge 0$ that, for any integer $i \ge 0$ and  any $\Inc^i$-invariant filtration $\Ic = (I_n)_{n \in \N}$, satisfying $\ind^i (\Ic) - i \le p$, the bigraded rational Hilbert series $H_{\Ic} (s, t)$ is rational, as desired.

Assume $p = 0$, that is, $\ind^{i} (\Ic) - i \le 0$.  Then we conclude by Lemma \ref{lem: p is zero}.

Let $p \ge 1$.  Now we use a second induction on $q \ge 0$ to show: if, for some integer $i \ge 0$,  an $\Inc^i$-invariant filtration $\Ic = (I_n)_{n \in \N}$  satisfies $r - i = \ind^i (\Ic) - i \le p$ and
\[
 \sum_{j = 0}^{e^+ (I_r)}  \dim_K [K[X_r]/I_r]_j \le q,
\]
then the  bigraded  Hilbert series $H_{\Ic} (s, t)$ is rational with  properties as claimed in the theorem. We refer to the above left-hand side as the \emph{$q$-invariant} of $\Ic$, that is,
\[
q (\Ic) = \sum_{j = 0}^{e^+ (I_r)}  \dim_K [K[X_r]/I_r]_j.
\]

By the first induction, we may assume $\ind^i (\Ic) - i = p$.  To begin the second induction, assume  $q = 0$. Then $I_r = K[X_r]$, and thus $K[X_n]/I_n = 0$ for each $n \ge r$. Hence
\[
H_{\Ic} (s, t) = \sum_{n = 0}^{r-1} H_{K[X_n]/I_n} (t) \cdot s^n = \sum_{n = 0}^{r-1} \frac{g_n(t)}{(1-t)^{d_n}} s^n,
\]
where $d_n = \dim K[X_n]/I_n$, and each $g_n (1) > 0$. Thus, using the notation of the theorem, $H_{\Ic} (s, t)$
is of the desired form, where we choose  the denominator as $(1-t)^a$ with
\[
a = \max \{d_n \; | \; 0 \le n < r\}, \; b= 0,
\]
and the coefficients of $0 \neq g(s, 1) =  \sum_{d_n = a} g_n (1)  s^{n}$ are non-negative.

Let $q \ge 1$, and assume that $q (\Ic) = q$. According to Corollary \ref{cor:index change when increasing i}, the sequence $\Ic$ is an $\Inc^{i+1}$-invariant filtration with $\ind^{i+1} (\Ic) \le r+1$. If $\ind^{i+1} (\Ic) \le r$, then $H_{\Ic} (s, t)$ has the desired form by induction on $p$.

Assume $\ind^{i+1} (\Ic) = r+1$.  We want to apply Lemma \ref{lem:multiple tables}.
For non-negative integers $e_1,\ldots,e_c$, consider a  filtration
\[
(\Ic : x_{1,i+1}^{e_1} \cdots x_{c,i+1}^{e_c},\ x_{1,i+1},\ldots,x_{c,i+1})_{n \in \N}
\]
whose $n$-th ideal is $\langle I_n  : x_{1,i+1}^{e_1} \cdots x_{c,i+1}^{e_c}, x_{1,i+1},\ldots,x_{c,i+1}\rangle $ if $n \ge r+1$ and $\langle 0\rangle$ if $1 \le n \le r$, as considered in Lemma \ref{lem:multiple tables}, where the role of  $r$ there is now taken by the number $r+1$.
By Lemma \ref{lem:multiple tables},  $(\Ic : x_{1,i+1}^{e_1} \cdots x_{c,i+1}^{e_c},\ x_{1,i+1},\ldots,x_{c,i+1})$ is $\Inc^{i+1}$-invariant with $(i+1)$-stability index $r+1$.
Assume that each exponent $e_k$ satisfies $0 \le e_k \le d$, where
\[
d =  \max\{ e \in \N_0 \; | \; x_{k, i+1}^e \text{ divides some monomial in } G(I_r) \text{ for some } k \in [c]\} .
\]
Since
\[
\sigma_{i} (I_r) \subset I_{r+1} \subset  I_{r+1} : x_{1,i+1}^{e_1} \cdots x_{c,i+1}^{e_c}
\]
 and, using $r +1 = p+ i+1 > i+1$,
\[
K[X_{r+1}]/\la \sigma_{i} (I_r), \ x_{1,i+1},\ldots,x_{c,i+1} \ra \cong K[X_r]/I_r
\]
we obtain, for each integer $j$,
\[
\dim_K [K[X_{r+1}]/\la I_{r+1} : x_{1,i+1}^{e_1} \cdots x_{c,i+1}^{e_c},\ x_{1,i+1},\ldots,x_{c,i+1} \ra]_j \le \dim_K [K[X_r]/I_r]_j.
\]
By $\Inc^i$-stability, we have $e^+(I_{r+1}) = e^+ (I_r)$, and thus
\[
e^* = e^+ (\la I_{r+1} : x_{1,i+1}^{e_1} \cdots x_{c,i+1}^{e_c},\ x_{1,i+1},\ldots,x_{c,i+1} \ra) \le e^+ (I_r).
\]
Hence, we get
\begin{equation}
    \label{eq:q estimate}
\sum_{j = 0}^{e^*} \dim_K [K[X_{r+1}]/\la I_{r+1} : x_{1,i+1}^{e_1} \cdots x_{c,i+1}^{e_c},\ x_{1,i+1},\ldots,x_{c,i+1}\ra]_j    \le  \sum_{j = 0}^{e^+ (I_r)}  \dim_K [K[X_r]/I_r]_j = q.
\end{equation}
Notice that the left-hand side is $q (\la \Ic : x_{1,i+1}^{e_1} \cdots x_{c,i+1}^{e_c},\ x_{1,i+1},\ldots,x_{c,i+1}  \ra)$.

If the inequality is strict, then we conclude by induction on $q$ that the $\Inc^{i+1}$-filtration $\la \Ic : x_{1,i+1}^{e_1} \cdots x_{c,i+1}^{e_c},\ x_{1,i+1},\ldots,x_{c,i+1}\ra$ has a  rational Hilbert series, as desired. Otherwise, that is, if we have equality in \eqref{eq:q estimate},  it follows that
\begin{equation*}
\la I_{r+1} : x_{1,i+1}^{e_1} \cdots x_{c,i+1}^{e_c},\ x_{1,i+1},\ldots,x_{c,i+1}\ra = \la \sigma_i (I_r), \ x_{1,i},\ldots,x_{c,i} \ra.
\end{equation*}
    Now applying  Lemma \ref{lem:shifted filtrations} to the right-hand side  gives
\begin{equation}
   \label{eq:shift}
H_{(\Ic : x_{1,i+1}^{e_1} \cdots x_{c,i+1}^{e_c},\ x_{1,i+1},\ldots,x_{c,i+1})} (s, t) =  s \cdot H_{\Ic} (s, t) + \frac{g_{\be} (s, t)}{(1-t)^{r c}},
\end{equation}
where $\be = (e_1,\ldots,e_c)$, \  $g_{\be} (s, t)\in \Z[s, t]$,  and $g_{\be} (s, 1) = s^r$. To simplify notation put
\[
\la \Ic_{\be} \ra = \la \Ic : x_{1,i+1}^{e_1} \cdots x_{c,i+1}^{e_c},\ x_{1,i+1},\ldots,x_{c,i+1} \ra.
\]
We now apply Lemma \ref{lem:multiple tables}.
Using Equation \eqref{eq:shift} for all filtrations with $q$-invariant equal to $q$, we obtain
\begin{equation}
    \label{eq:recursion}
\begin{split}
 H_{\Ic} (s, t) = & \frac{h(s, t)}{(1-t)^{ (r+1) c}} +  \frac{1}{(1-t)^c} \cdot
\sum_{\be = (e_1,\ldots,e_c)\in \Z^c,\ 0 \le e_l \le d \atop   q (\la \Ic_{\be} \ra ) < q }
f_{\be} (t) \cdot H_{\langle \Ic_{\be} \rangle} (s, t) \\[1ex]
& + \frac{1}{(1-t)^c} \cdot \sum_{\be = (e_1,\ldots,e_c)\in \Z^c,\ 0 \le e_l \le d \atop    q (\la \Ic_{\be} \ra ) = q}
f_{\be} (t) \cdot \left [ s \cdot H_{\Ic} (s, t) + \frac{g_{\be} (s, t)}{(1-t)^{r c}}   \right ],
\end{split}
\end{equation}
where $h(s, 1) = - d^{c-1} s^{r}$ and, using the notation of Lemma \ref{lem:multiple tables},
\[
f_{\be} (t)  = t^{|\be|} (1-t)^{\delta (\be)}.
\]
Collecting terms, we get
\begin{equation}
     \label{eq:one side}
H_{\Ic} (s, t) \cdot \left [1 - \frac{s}{(1-t)^c} \Tilde{f} (t) \right ] =  \frac{\Tilde{g}(s, t)}{(1-t)^{(r+1)c}} + \frac{1}{(1-t)^c} \cdot
\sum_{\be = (e_1,\ldots,e_c)\in \Z^c,\ 0 \le e_l \le d \atop    q (\la \Ic_{\be} \ra ) < q }
f_{\be} (t) \cdot H_{\langle \Ic_{\be} \rangle} (s, t),
\end{equation}
where
\[
\Tilde{f}(t) = \sum_{\be = (e_1,\ldots,e_c)\in \Z^c,\ 0 \le e_l \le d \atop    q (\la \Ic_{\be} \ra ) = q} t^{|\be|} (1-t)^{\delta (\be)} = (1-t)^{\tilde{c}} f (t)
\]
with $0 \le \tilde{c} \le c$, \ $f(t) \in \Z[t]$, \ $f(1) \ge 0$. Moreover,  $f(1)$ is positive, unless there is no multi-index $\be$ such that $q (\la \Ic_{\be} \ra ) = q$, in which case $\Tilde{f} = 0$. Furthermore,
\[
\Tilde{g} (s, t) = h(s, t) + \sum_{\be = (e_1,\ldots,e_c)\in \Z^c,\ 0 \le e_l \le d \atop    q (\la \Ic_{\be} \ra ) = q} f_{\be} (t) g_{\be} (s, t) \in \Z[s, t]
\]
and
\[
\Tilde{g} (s,1) = -d^{c-1} s^r + \sum_{\be = (e_1,\ldots,e_c)\in \Z^c,\ 0 \le e_l \le d \atop    q (\la \Ic_{\be} \ra, \delta (\be) = 0 } s^r.
\]
 Since each bigraded  Hilbert series appearing on the right-hand side of Equation \eqref{eq:one side} is rational of the desired form by induction, we conclude that $H_{\Ic} (s, t)$ is rational and that it  has the claimed shape.  Here we use in particular that the degree of the numerator polynomial of any $H_{\langle \Ic_{\be} \rangle} (s, t)$, after  evaluation at $t = 1$,  is at most $r+1 + q-1= r+q$.

 It remains to establish the estimates on $a$ and $b$.
Consider $b$.
If $q = 0$, then we have seen above that we  can always achieve $b = 0$.
By induction, each $H_{\Ic_{\be}} (s, t)$   on the right-hand side of Equation \eqref{eq:one side} has at most $\frac{(d+1)^{(q-1) c} - 1}{(d+1)^c - 1}$ factors other than $(1-t)$ in the denominator. There are at most $(d+1)^c$ such Hilbert series.  Hence, the number of factors other than $(1-t)$ in the denominator of $H_{\Ic} (s, t)$ is at most $(d+1)^c \cdot \frac{(d+1)^{(q-1) c} - 1}{(d+1)^c - 1} + 1 = \frac{(d+1)^{q c} - 1}{(d+1)^c - 1}$, as claimed.

Finally, we estimate $a$. If $p \le 0$ or $q = 0$, then we have shown above that $a \le (r-1) c$. By induction on $q$, the right-hand side of Equation \eqref{eq:one side}  can be written as a rational function whose  exponent of $(1- t)$ in the denominator is at most $[r + 2(q-1)) + 1] c$. Solving for $H_{\Ic}$ in this equation, we see that  $a \le (r-1+2q) c$, which completes the argument.
\end{proof}


\section{Consequences for graded ideals and algebras}
\label{sec:consequences}

Combining our results, we establish the rationality of bigraded Hilbert series of any $\Inc^i$-invariant filtration. Then we derive asymptotic properties of invariants of the ideals in such a chain. We also state the consequences for ideals that are invariant under the action of a symmetric group. Several  examples illustrate our results.
Throughout this section,  when considering initial ideals,   we use the  lexicographic order $\preceq$ on $K[X]$ or $K[X_n]$ induced by the ordering of the variables $x_{i, j} \preceq x_{i', j'}$ if either $i < i'$ or  $i = i'$ and $j < j'$.

\begin{lem}
   \label{lem:pass initial}
For any $i \in \N_0$, let $\Ic = (I_n)_{n \in \N}$ be an $\Inc^i$-invariant filtration. Then the sequence $\ini_{\preceq} (\Ic) = (\ini_{\preceq} (I_n))_{n \in \N}$ also is an $\Inc^i$-invariant filtration and
\[
 \ind^i (\Ic) \le \ind^i (\ini_{\preceq} (\Ic)).
\]
\end{lem}

\begin{proof}
This is clear if $\Ic$ is trivial, so assume that $\Ic$ is non-trivial. Using $\ini_{\preceq} \left (\pi(g) \right ) = \pi \left ( \ini_{\preceq} (g) \right )$ for all $\pi \in \Inc^i$ and all $g \in K[X]$, one checks that $\ini_{\preceq} (\Ic)$  is indeed an $\Inc^i$-invariant filtration.

Let $r$ be the least integer $n \in \N$ such that there is a finite $\Inc^i$-Gr\"obner basis $B$ with respect to $\preceq$ of $I$, which is contained  in $K[X_n]$.
Then $\ind^i (\Ic) \le r$ by Remark \ref{rem:index est}. Moreover, the choice of $B$ implies $\ind^i (\ini_{\preceq} (\Ic)) = r$, and we are done.
\end{proof}

The desired result for $\Inc^i$-invariant ideals is:

\begin{prop}
   \label{prop:main thm}
Let $ i \ge 0$ and consider any  $\Inc^i$-invariant  filtration of graded ideals  $\Ic = (I_n)_{n \in \N}$.  
Then the bigraded Hilbert series $H_{\Ic} (s, t)$  of $\Ic$ is a  rational function  of the form
\[
H_{\Ic} (s, t) = \frac{g(s, t)}{(1-t)^a \cdot \prod_{j =1}^b [(1-t)^{c_j} - s \cdot f_j (t)]},
\]
where $a, b, c_j$ are non-negative integers with $c_j \le c$, \ $g (s, t) \in \Z[s, t]$, and each $f_j (t)$ is a polynomial in $\Z[t]$ satisfying $f_j (1) >  0$.
\end{prop}

\begin{proof}
Using the well-known result that  $K[X_n]/I_n$  and $K[X_n]/\ini_{\preceq} (I_n))$ have the same the Hilbert series, this follows by combining Lemma \ref{lem:pass initial} and Theorem \ref{thm:rat Hilb series, rev}.
\end{proof}

In order to illustrate this result we discuss some examples.

\begin{ex}
    \label{exa:tensor}
Every graded ideal $J$ in a polynomial ring $S = K[Y_1,\ldots,Y_c]$ in $c$ variables gives  rise to an in $\Inc$-invariant ideal $I \subset K[X]$ and $\Inc$-invariant filtration $\Ic$. Indeed,   let $\phi: S \to K[X_1]$ be the ring isomorphism, defined by $\phi (Y_i) = X_{i, 1}$, and define a sequence of ideals $\Ic = (I_n)_{n \in \N}$ by
\[
I_n = \begin{cases}
\phi (J) & \text{ if } n = 1,\\
\la \Inc_{1, n} (I_1) \ra & \text{ if } n > 1.
\end{cases}
\]
Then $\Ic$ is an $\Inc$-invariant chain, and $I = \bigcup_{n \in \N} I_n$ is the $\Inc$-closure of $I_1 = \phi (J)$ (see Lemma \ref{lem:equivfilt} with $r=1$). For example, if $J = \la Y_1 Y_2\ra$ and $c = 2$, then
$I = \la X_{1, n} X_{2, n} \; | \; n \in \N \ra$.

Coming back to the general set-up of the example, the bigraded Hilbert series of $\Ic$ is
\[
 H_{\Ic} (s, t) = \frac{(1-t)^d}{(1-t)^d - s \cdot f(t)},
\]
where $H_{S/J} (t) = \frac{f(t)}{(1-t)^d}$ is the reduced Hilbert series of $S/J$, that is, $d = \dim S/J \le c$ and $f(1) > 0$ is the degree of $J$.
Indeed, for each $n \ge 1$, there is a graded ring isomorphism
\[
K[X_n]/I_n \cong (S/J)^{\otimes n} = S/J \otimes_K \cdots \otimes_K S/J.
\]
Thus, we get for the Hilbert series of $K[X_n]/I_n$:
\[
H_{K[X_n]/I_n} (t) = (H_{S/J} (t))^n = \left (\frac{f(t)}{(1-t)^d} \right)^n.
\]
Using the geometric series our claim follows.
Notice, in particular, that the degree of $f$ can be arbitrarily large and that the coefficients of $f$ can be negative. Thus, one cannot  hope to have a general  stronger result about the polynomials $f_j$ occurring in Proposition \ref{prop:main thm}.
\end{ex}

Determinantal ideals to generic matrices with a varying number of columns also give rise to an $\Inc$-invariant filtration in a different way. For simplicity, we consider  ideals generated by 2-minors. These ideals arise naturally in the study of two-way contingency tables (see, e.g, \cite[Example 4.2]{HS}).

\begin{ex}
   \label{exa:2-minors}
Fix an integer $c \ge 2$, and consider a sequence of ideals $\Ic = (I_n)_{n \in \N}$, defined by
\[
I_n = \begin{cases}
\la 0 \ra & \text{ if } n =1,\\
I_2 (X_{c \times n}) & \text{ if } n \ge 2,
\end{cases}
\]
where $X_{c \times n}  = (x_{i, j})$ denotes a generic $c \times n$ matrix in variables $x_{i, j}$.
The sequence $\Ic$ is $\Inc$-invariant, as observed, for example, in \cite[Example 4.2]{HS}.  Indeed, if a 2-minor $f$ is obtained by using columns $i$ and $j$, then, for each $\pi \in \Inc$, the polynomial $\pi (f)$ is the minor defined by columns $\pi (i)$ and $\pi (j)$ and the same rows that give $f$. If $n \ge 2$, then, see, e.g., \cite[Example 5.10]{CorsoNagel}, the Hilbert series of $K[X_n]/I_n$ is
\[
H_{K[X_n]/I_n} (t) = \frac{1}{(1-t)^{c+n-1}} \cdot \sum_{j=0}^{c-1} \binom{c-1}{j} \binom{n-1}{j} t^j.
\]
Thus, using the formula
\[
\sum_{n \ge j} \binom{n}{j} s^n = \frac{s^j}{(1-s)^{j+1}},
\]
we get for   the bigraded Hilbert series of $\Ic$:
\begin{align*}
H_{\Ic} (s, t) & = 1 + \sum_{n \ge 1} \left [ \frac{1}{(1-t)^{c+n-1}} \sum_{j=0}^{c-1} \binom{c-1}{j} \binom{n-1}{j} t^j \right ] s^n \\
& = 1 + \frac{s}{(1-t)^c} \sum_{j=0}^{c-1} \left [  \binom{c-1}{j} t^j  \cdot \sum_{n-1 \ge j} \binom{n-1}{j} \left ( \frac{s}{1-t} \right )^{n-1} \right ]  \\
& = 1 + \frac{s}{(1-t)^c} \sum_{j=0}^{c-1} \left [  \binom{c-1}{j} t^j  \cdot \frac{(1-t) s^j}{(1-t-s)^{j+1}} \right ] \\
& = 1 + \frac{s}{(1-t)^{c-1} (1-t-s)^c} \sum_{j=0}^{c-1} \binom{c-1}{j} (s t)^j (1-t-s)^{c-1-j} \\
& = 1 + \frac{s}{(1-t)^{c-1} (1-t-s)^c} (s t + 1 - t - s)^{c-1} \\
& = 1 + \frac{s}{(1-t)^{c-1} (1-t-s)^c} ([1-t] [1-s])^{c-1} \\
& = \frac{ (1-t-s)^c + s( 1- s)^{c-1}}{(1-t-s)^c}.
\end{align*}
Note that, after evaluation at $t=1$, the numerator polynomial has degree $c-1$, whereas the stability index of $\Ic$ equals 2.
\end{ex}

We now consider other monoids that act on $K[X]$.
Denote by $\Sym (n)$ the group of bijections $\pi: [n] \to [n]$.
The group  $\Sym (n)$ is naturally embedded into  $\Sym (n+1)$ as the stabilizer of $\{n+1\}$. Set
\[
\Sym (\infty) = \bigcup_{n \in \N} \Sym (n).
\]
It acts on $K[X]$  by
\[
\pi \cdot x_{i, j} = x_{i, \pi (j)}.
\]
Observe that, for each non-trivial $\pi \in \Sym (\infty)$ and each $n$, the induced map $\pi: X_n \to X_m$ is injective, where $m = \max \{\pi (j) \; | \; 1 \le j \le n\}$.
Note that $\Inc \nsubseteq \Sym(\infty)$. However, there is a well-known inclusion of orbits (see, e.g., \cite{HS}).

\begin{lem}
    \label{lem:comp orbits}
For each polynomial $f \in K[X_m]$ and any pair of positive integers $m \le n$, one has
\[
\Inc_{m, n} \cdot f \subset \Sym(n) \cdot f.
\]
In particular, $\Inc \cdot f \subset \Sym(\infty) \cdot f$ for each $ f \in K[X]$. Thus, each $\Sym(\infty)$-invariant ideal in $K[X]$ also  is $\Inc$-invariant.
\end{lem}

\begin{proof}
Let $\pi \in \Inc_{m, n}$. Choose some $\sigma \in \Sym(n)$ satisfying 
\[
\sigma (j) = \begin{cases}
\pi (j) & \text{ if } j \le m,  \\
j & \text{ if } \pi (m) < j \le n.
\end{cases}
\]
It follows $\pi \cdot f = \sigma \cdot f$, which implies the claims.
\end{proof}

\begin{ex}
   \label{exa:inc but not sym-invariant}
(i)
Let $c =2$, and consider the ideal
\[
I = \la x_{1, i} x_{2, j} \; | \: i, j \in \N \ra \subset K[X].
\]
It is $\Inc$-invariant and $\Sym (\infty)$-invariant. It is minimally generated by the $\Sym (\infty)$ orbits of $x_{1,1} x_{2,1}$ and $x_{1,1} x_{2,2}$. However, one needs three $\Inc$-orbits to generate $I$. In fact, the $\Inc$-orbits of $x_{1,1} x_{2,1}, \ x_{1,1} x_{2,2},$ and $x_{1,2} x_{2,1}$ give a minimal generating set of $I$.

(ii) There are more $\Inc$-invariant ideals than $\Sym (\infty)$-invariant ideals as is easily seen. For example,  let $c = 1$ and consider the ideal
\[
J =  \la x_{1, i} x_{1, j} \; | \: i, j \in \N  \text{ and } i+2 \le j \ra \subset K[X].
\]
It is an $\Inc$-invariant ideal generated by the orbit of $x_{1, 1} x_{1, 3}$. However, the ideal $J$ is \emph{not} $\Sym (\infty)$-invariant.
\end{ex}

We are ready to establish one of our main results.

\begin{thm}
   \label{thm:main thm sym}
Let   $\Ic = (I_n)_{n \in \N}$ be an $\Sym(\infty)$-invariant  filtration of graded ideals, that is, a sequence of ideals $I_n \subset K[X_n]$ such that
\[
\Sym (n) (I_m) \subseteq I_n   \quad \text{ whenever } m \le n.
\]
Then its bigraded Hilbert series is rational and can be written as in Proposition \ref{prop:main thm}.
\end{thm}

\begin{proof}
According to Lemma \ref{lem:comp orbits}, $\Ic$ is an $\Inc$-invariant filtration. Thus, Proposition \ref{prop:main thm} yields the assertion.
\end{proof}

\begin{rem}
An analogous statement is true for any $\Sym^i (\infty)$-filtration, where $\Sym^i (\infty)$ is the subset of elements $\pi \in \Sym (\infty)$ that fix  each integer in $[i]$.
\end{rem}

As an application of our rationality result for bigraded Hilbert series,  we derive information on the asymptotic behavior of some invariants.

\begin{thm}
   \label{thm:asymp dim}
Let $\Ic =  (I_n)_{n \in \N}$ be an $\Inc^i$-invariant or $\Sym (\infty)$-filtration. Then there are integers $A, B, M, L$ with $0 \le A \le c$, $M \ge 0$, and $L \ge 0$ such that, for all $n \gg 0$,
\[
\dim K[X_n]/I_n = A n + B
\]
and the  limit  of $\frac{\deg I_n}{M^{n } \cdot  n^{L}}$ as $n \to \infty$ exists and is equal to a positive rational number.
\end{thm}

Our argument is constructive. The proof shows how to read off 
the integers $A, B, M, L$ as well as the limit from the equivariant Hilbert series. 

\begin{proof}[Proof of Theorem \ref{thm:asymp dim}]
According to Proposition \ref{prop:main thm} and Theorem \ref{thm:main thm sym}, the bigraded Hilbert series of $\Ic$ can be written as
\[
H_{\Ic} (s, t) = \frac{g(s, t)}{(1-t)^a \cdot \prod_{j =1}^m [(1-t)^{c_j} - s \cdot f_j (t)]^{b_j}},
\]
where $a, m,  c_j$ are non-negative integers with $c_j \le c$, \ $g (s, t) \in \Z[s, t]$,  each $f_j (t)$ is a polynomial in $\Z[t]$ satisfying $f_j (1) >  0$, each integer $b_j$ is positive, provided $m > 0$, and
\[
[(1-t)^{c_j} - s \cdot f_j (t)] \neq [(1-t)^{c_k} - s \cdot f_k (t)]
\]
if $j \neq k$ are in $[m]$. We also assume that $g(s, t)$ is not divisible in $\Q[s, t]$ by any of the factors $[(1-t)^{c_j} - s \cdot f_j (t)]$.

If $m = 0$, then $H_{\Ic} (s, t)$ is a polynomial in $\Q(t)[s]$, which implies $I_n = K[X_n]$ for all $n \gg 0$. It follows that the assertions are true for $A = B = M = L = 0$.

Assume now $m> 0$ and set
\[
A = \max \{c_1,\ldots,c_m\}.
\]
Thus $0 \le A \le c$. Rewrite $H_{\Ic} (s, t)$ as
\[
H_{\Ic} (s, t) = \frac{g(s, t) \cdot (1-t)^{-a - \sum_{j=1}^m c_j b_j}}{\prod_{j =1}^m [1- \frac{f_j (t)}{(1-t)^{c_j}} s]^{b_j}}.
\]
Using partial fractions over the function field $\Q(t)$ and clearing denominators, there are polynomials $r(t), r_{j, k} (t) \in \Z[t]$, a polynomial $\widetilde{g} (s, t) \in \Q(t)[s]$, and an integer $\gamma$ such that
\[
H_{\Ic} (s, t) = \widetilde{g} (s, t) +  \sum_{j=1}^m \sum_{k = 1}^{b_j}  \frac{\frac{{r}_{j, k} (t) \cdot (1-t)^{\gamma}}{r(t)}}{[1- \frac{f_j (t)}{(1-t)^{c_j}} s]^{k}} =
\widetilde{g} (s, t) + \frac{(1-t)^{\gamma}}{r(t)} \cdot  \sum_{j=1}^m \sum_{k = 1}^{b_j}  \frac{{r}_{j, k} (t)}{[1- \frac{f_j (t)}{(1-t)^{c_j}} s]^{k}}
\]
and $r (1) \neq 0$. Notice that each $r_{j, b_j} (t)$ is not the zero polynomial. Setting
\[
G(s, t) = \sum_{j=1}^m \sum_{k = 1}^{b_j}  \frac{{r}_{j, k} (t)}{[1- \frac{f_j (t)}{(1-t)^{c_j}} s]^k},
\]
we get
\begin{equation}
    \label{eq:partial frac}
H_{\Ic} (s, t) = \widetilde{g} (s, t) +  \frac{(1-t)^{\gamma}}{r(t)} \cdot G(s, t).
\end{equation}

Using binomial series, we obtain
\begin{align*}
G(s, t) & = \sum_{j=1}^m \sum_{k = 1}^{b_j} {r}_{j, k} (t) \cdot \sum_{n \ge 0} \binom{n+k-1}{k-1}  \frac{f_j (t)^n}{(1-t)^{c_j n}} s^n \\
& =  \sum_{n \ge 0} \frac{1}{(1-t)^{A n}} h_n (t)  s^n,
\end{align*}
where
\[
h_n (t) = \sum_{j=1}^m \sum_{k = 1}^{b_j}  \binom{n+k-1}{k-1} {r}_{j, k} (t) f_j (t)^n (1-t)^{(A - c_j) n} \in \Q[t].
\]
Now sort the elements of $\{f_j (1) \; | \; j \in [m] \text{ and } c_j = A\}$:
\[
M_1 > M_2 > \cdots > M_v,
\]
where $1 \le v \le m$. Set
\[
\Sc_l = \{ j \in [m] \; | \; c_j = A  \text{ and }  f_j (1) = M_l\}.
\]
Separating  the summands in $h_n (t)$ with $c_j = A$ and ordering them by the size of  $f_j (1)$, we get
\begin{equation}
    \label{eq:dominant part}
h_n (t)   =  \sum_{l=1}^v  p_{l, n}  (t) + \sum_{j \in [m], \ c_j < A } \ \ \sum_{k = 1}^{b_j}  \binom{n+k-1}{k-1} {r}_{j, k} (t) f_j (t)^n (1-t)^{(A - c_j) n}
\end{equation}
where
\[
p_{l, n} (t) = \sum_{j \in \Sc_l} \sum_{k = 1}^{b_j}  \binom{n+k-1}{k-1} {r}_{j, k} (t) f_j (t)^n.
\]
Let $l \in [v]$. In order to study $p_{l, n} (1)$ for $n \gg 0$  we consider the power series $\sum_{n \ge 0} p_{l, n} (t) s^n$. Using binomial series again, we get
\[
\sum_{n \ge 0} p_{l, n} (t) s^n = \sum_{j \in \Sc_l} \ \sum_{k = 1}^{b_j} \frac{r_{j, k} (t)}{[1- f_j(t) s]^k}.
\]
Applying Lemma \ref{lem:app key power series} of the appendix with $a = 1$, we obtain  non-negative  integers $\delta_l$ and $L_l$ such that, for $n \gg 0$,
\[
p_{l, n} (t)  = (1-t)^{\delta_l} \cdot q_{l, n} (t)
\]
with polynomials $q_{l, n}  (t) \in \Z[t]$ such that the limit of $\frac{q_{l, n}  (1)}{M_l^{n } \cdot  n^{L_l}}$ as $n \to \infty$ exists and is equal to a non-zero rational number. Since $M_l \ge 1$ this includes the fact that $q_{l, n}  (1) \neq 0$ for $n \gg 0$. Set now
\[
\delta = \min \{\delta_l \; | \; l \in [v]\}
\]
and
\[
l^* = \max \{l \in [v] \; | \; \delta = \delta_l\}.
\]
Using this notation,  if $n \gg 0$ then  Equation \eqref{eq:dominant part} gives  polynomials
\begin{align*}
\widetilde{h}_n (t) & = \frac{1}{(1-t)^\delta} h_n (t) \\
&  = \sum_{l=1}^v  q_{l, n}  (t) (1-t)^{\delta_l - \delta} + \sum_{j \in [m], c_j < A } \ \sum_{k = 1}^{b_j}  \binom{n+k-1}{k-1} {r}_{j, k} (t) f_j (t)^n (1-t)^{(A - c_j) n -  \delta}.
\end{align*}
It follows for $n \gg 0$
\[
\widetilde{h}_n (1) = \sum_{l \in [v], \delta_l = \delta } q_{l, n}  (1),
\]
and thus the following limit exists and satisfies
\begin{equation}
   \label{eq:limit}
\lim_{n \to \infty} \frac{\widetilde{h}_n (1)}{M_{l^*}^{n } \cdot  n^{L_{l^*}}} = \sum_{l \in [v], \delta_l = \delta } \lim_{n \to \infty} \frac{q_{l, n}  (1)}{M_l^{n } \cdot  n^{L_l}} \cdot \lim_{n \to \infty} \left [ \left  (\frac{M_l}{M_{l^*}} \right)^n n^{L_{l^*} - L_l} \right] = \lim_{n \to \infty} \frac{q_{l^*, n}  (1)}{M_{l^*}^{n } \cdot  n^{L_{l^*}}} \neq 0
\end{equation}
because $M_{l^*} > M_l$ if $l \neq l^*$  by the choice of $l^*$ and the ordering of the integers $M_l$. 

Hence, using  Equation \ref{eq:partial frac}, the formula for $G(s, t)$, and $\widetilde{h}_n (t)  = \frac{1}{(1-t)^\delta} h_n (t)$, we obtain 
\[
H_{\Ic} (s, t) = \widetilde{g} (s, t) +  \frac{1}{r(t)} \cdot \sum_{n \ge 0} \frac{1}{(1-t)^{A n - (\gamma + \delta)}} \widetilde{h}_n (t) s^n,
\]
where $\widetilde{h}_n (1) \neq 0$ if $n \gg 0$.  Setting
\[
M = M_{l^*}, \; L = L_{l^*} \; \text{ and } B = - (\gamma + \delta)
\]
and  comparing coefficients if follows for $n \gg 0$
\[
H_{K[X_n]/I_n} (t) =  \frac{1}{r(t)} \cdot  \frac{1}{(1-t)^{A n +B}} \widetilde{h}_n (t),
\]
and thus
\[
\dim K[X_n]/I_n = An + B
\text{ and }
\deg I_n = \frac{1}{r(1)} \cdot \widetilde{h}_n (1).
\]
Now Equation \eqref{eq:limit} yields
\[
\lim_{n \to \infty} \frac{\deg I_n}{M^{n } \cdot  n^{L}} = \frac{1}{r(1)} \cdot   \lim_{n \to \infty} \frac{q_{l^*, n}  (1)}{M^{n } \cdot  n^{L}} > 0,
\]
which completes the argument.
\end{proof}

Often one can make the numbers $A, B, M, L$ and the limit more explicit.

\begin{cor}
   \label{cor:asymp dim special}
Let $\Ic =  (I_n)_{n \in \N}$ be an $\Inc^i$-invariant or $\Sym (\infty)$-filtration. Write its bigraded Hilbert series as in Proposition \ref{prop:main thm} and
\[
g (s, t) = \sum_{j = 0}^N (1-t)^{e_j} g_j (t) s^j,
\]
where each polynomial $g_j (t) \in \Z[t]$ satisfies $g_j (1) \neq 0$ or $g_j (t) = 0$. Set
\begin{align*}
A & = \max \{c_1,\ldots,c_b\},  \\
B & = a +  \sum_{k = 1}^b c_k + \max \{ -A j - e_j \: | \; 0 \le j \le N \text{ and }  g_j (1) \neq 0\}.
\end{align*}
Then one has:
\begin{itemize}
\item[(a)] For all $n \gg 0$,
\[
\dim K[X_n]/I_n \leq A n + B.
\]

\item[(b)]  Possibly after reindexing, we may assume
\[
M = f_1 (1) = \cdots = f_{l+1} (1) > f_{l+2} \ge \cdots \ge f_{\widetilde{b}} (1),
\]
where $\widetilde{b}$ is defined by $A = c_1 = \cdots = c_{\widetilde{b}} > c_{\widetilde{b} + 1}  \ge \cdots \ge c_b$. Furthermore, suppose
\[
\sum_{j =0 \atop A j + e_j = - B + a + \sum_{j = 1}^b c_j}^N g_j (1) \cdot M^j \neq 0.
\]
Then $\dim K[X_n]/I_n = A n + B$ for all $n \gg 0$ and
\[
\lim_{n \to \infty} \frac{\deg I_n}{n^{\ell} \cdot  M^{n }}= \frac{M^{\widetilde{b} - 1 - l}}{l!} \cdot \prod_{j = l+2}^{\widetilde{b}}  \left ( \sum_{j =0 \atop A j + e_j = - B + a + \sum_{j = 1}^b c_j}^N g_j (1) \cdot M^j  \right ) > 0.
\]
\end{itemize}
\end{cor}

\begin{proof}
The geometric series gives
\[
\frac{1}{(1-t)^{c_j} - f_j (t) s} = \frac{1}{(1-t)^{c_j}} \sum_{n \ge 0} \frac{f_j (t)^n}{(1-t)^{c_j n}} s^n.
\]
Thus, Proposition \ref{prop:main thm} implies
\begin{align*}
H_{\Ic} (s, t) &  = \frac{1}{(1-t)^{a + \sum c_j}} \left (\sum_{j=0}^N (1-t)^{e_j} g_j (t) s^j \right ) \cdot \sum_{n \ge 0} \left [  \sum_{n_1 + \cdots + n_b = n, \ n_j \in \N_0} \frac{f_1(t)^{n_1} \cdots f_b(t)^{n_b}}{(1-t)^{\sum c_j n_j}} \right ] s^n \\
& = \frac{1}{(1-t)^{a + \sum c_j}}  \sum_{n \ge 0} \left [ \sum_{j=0}^{\min \{n, N\}} (1-t)^{e_j} g_j (t) \cdot  \sum_{n_1 + \cdots + n_b = n - j, \ n_k \in \N_0} \frac{f_1(t)^{n_1} \cdots f_b(t)^{n_b}}{(1-t)^{\sum c_k n_k}}  \right ] s^n.
\end{align*}
Our ordering $A = c_1 = \cdots = c_{\widetilde{b}} > c_{\widetilde{b} + 1}  \ge \cdots \ge c_b$ now implies, for $n \ge N$,
\[
\max \{ \sum_{k=1}^b c_k n_k \; | \; n_1 + \cdots + n_b = n- j, \ n_k \in \N_0\} = A (n-j).
\]
Moreover,  the sum attains its maximum if and only if $n_{\widetilde{b} + 1} = \cdots = n_b = 0$. Hence there are polynomials $p_n (t) \in \Z[t]$ such that
\[
H_{\Ic} (s, t) = \sum_{n \ge 0} \frac{p_n (t)}{(1-t)^{A n + B}} s^n,
\]
where, for $n \gg 0$,
\[
p_n (1) = \sum_{j = 0 \atop A j + e_j = - B + a + \sum_{j = 1}^b c_j}^N g_j (1) \cdot \sum_{n_1 + \cdots + n_b = n - j, \ n_k \in \N_0} f_1(t)^{n_1} \cdots f_b(t)^{n_b}.
\]
It follows that $\dim K[X_n]/I_n \le A n + B$ for $n \gg 0$, proving (a). In fact, this is an equality of $p_n (1) \neq 0$ for all $n \gg 0$.

Set $D = \prod_{j=2}^{\widetilde{b}} (M - f_j (1))$. Then Lemma \ref{lem:sum estimate} gives
\[
\lim_{n \to \infty} \frac{\sum_{n_1 + \cdots + n_{\widetilde{b}} = n - j, \ n_k \in \N_0} f_1(t)^{n_1} \cdots f_{\widetilde{b}}(t)^{n_{\widetilde{b}}}}{M^{n-j} \cdot n^l} = \frac{M^{\widetilde{b} - 1 - l}}{D \cdot l!}.
\]
Hence, we obtain
\[
\lim_{n \to \infty} \frac{p_n (1)} {M^n \cdot n^l} = \frac{M^{\widetilde{b} - 1 - l}}{D \cdot l!} \cdot \left ( \sum_{j =0 \atop A j + e_j = - B + a + \sum_{j = 1}^b c_j}^N g_j (1) \cdot M^j  \right ),
\]
which is not zero by assumption. In particular, we conclude that $p_n (1) \neq 0$ for all $n \gg 0$, which implies $\dim K[X_n]/I_n =  A n + B$ and
$\deg I_n = p_n (1)$ for all $n \gg 0$. Hence, the above limit must be positive, and the argument is complete.
\end{proof}

Let us illustrate this result.

\begin{ex}
    \label{exa:illustrate estimates}
In Example \ref{exa:2-minors}, we obtained for the bigraded Hilbert series
\[
H_{\Ic} (s, t)  = \frac{ (1-t-s)^c + s( 1- s)^{c-1}}{(1-t-s)^c}.
\]
Using the notation of Corollary \ref{cor:asymp dim special}, we get $a= 0, b = c,  c_1 = \cdots =  c_b = 1$, and thus $A = 1$ as well as $M = 1$ and $l = c-1$. Furthermore, we have $N = c-1, \ g_0 (t) = 1, \ e_0 = c$, and $g_j (t) = (-1)^j [\binom{c}{j} (1-t)^{c-j} - \binom{c-1}{j-1}], \ e_j = 0$ if $1 \le j \le c-1$. It follows that
\[
- A j - e_j = -j - e_j = \begin{cases}
-c & \text{if } j = 0, \\
-j & \text{if }  1 \le j < c.
\end{cases}
\]
Hence we obtain $B = c-1$ and
\[
\sum_{j =0 \atop A j + e_j = - B + a + \sum_{j = 1}^b c_j}^N g_j (1) \cdot M^j = g_1 (1) = 1 \neq 0.
h\]
Thus, Corollary \ref{cor:asymp dim special} is applicable and correctly gives $\dim K[X_n]/I_n = n+c-1$ as well as $\lim_{n \to \infty} \frac{\deg I_n}{n^{c-1}} = \frac{1}{(c-1)!}$ for $n \gg 0$.
\end{ex}

Recall that the bigraded Hilbert series of an $\Inc^i$-invariant or $\Sym (\infty)$-invariant ideal $I$ of $K[X]$ is defined by using its saturated filtration, that is,
\[
H_{K[X]/I} (s, t) = \sum_{n \ge 0,\, j \ge 0} \dim_K [K[X_n]/(I \cap K[X_n])]_j \cdot s^n t^j.
\]

\begin{cor}
   \label{cor:Hilb ideal}
Let $I \subset K[X]$ be an   $\Inc^i$-invariant or $\Sym (\infty)$-invariant ideal. Then its
 bigraded Hilbert series is rational and can be written as in Proposition \ref{prop:main thm}.
\end{cor}

\begin{proof}
If $I \subset K[X]$ is an   $\Inc^i$-invariant or $\Sym (\infty)$-invariant ideal, then the sequence $\Ic = (I \cap K[X_n])_{n \in \N}$ is an $\Inc^i$-invariant or $\Sym (\infty)$-invariant filtration, respectively. Thus, Proposition \ref{prop:main thm}  and Theorem \ref{thm:main thm sym} prove the assertions.
\end{proof}

\begin{rem}
Consider an   $\Inc^i$-invariant or $\Sym (\infty)$-invariant ideal $I \subset K[X]$ and its saturated filtration $\Ic = (I_n)_{n \in \N}$.  Theorem \ref{thm:asymp dim} shows that the growth of the  dimensions of  $K[X_n]/I_n$ is dominated by an integer $A$ and the growth of the degrees of $I_n$ is dominated by a rational number, which is in fact a certain limit. This suggest to  define the dimension of $K[X]/I$ to be $A$ and the degree of $I$ to be the mentioned limit.
\end{rem}

\appendix
\section{}

Here we establish some technical results that are used in the body of the paper.

\begin{lem}
   \label{lem:pol coefficients}
Let $f_1 (t),\ldots,f_m (t) \in \R[t]$ be $m > 0$ distinct polynomials, and let $b_1,\ldots,b_m$ be any positive integers. Then
$\frac{1}{\prod_{j=1}^m [1 - f_j (t) s]^{b_j}} \in \R[t]\llbracket s \rrbracket$
admits a representation
\[
\frac{1}{\prod_{j=1}^m [1 - f_j (t) s]^{b_j}} = \sum_{n \ge 0} q_n (t) s^n,
\text{ where  } q_n (t)\in \R[t].
\]
\end{lem}

\begin{proof}
The binomial series shows that, for all integers $b \ge 0$, the inverse of  $[1 - f_j (t) s]^b$  in the ring of formal power series in $s$ with coefficients in $\R[t]$  is
\[
\frac{1}{[1 - f_j (t) s]^b}=
\sum_{n \ge 0} \binom{n+b -1}{b- 1} f_j (t)^n s^n.
\]
Using this fact and sorting by powers of s, the claim follows.
\end{proof}

Retaining its notation and assumptions, we use Lemma \ref{lem:pol coefficients} to establish the following result.

\begin{lem}
   \label{lem:app key power series}
Consider polynomials $r_{j, k} (t) \in \R[t]$, where $1 \le j \le m, \ 1 \le k \le b_j$ such that, for each $j \in [m]$, the polynomial $r_{j, b_j} (t)$ is not the zero polynomial. Assume that, for some $a \in \R$,
\[
f_1 (a) = \cdots = f_m (a) = M \neq 0.
\]
Then there are non-negative integers $\delta$ and $L$ and polynomials $q_n (t) \in \R[t]$ such that
\[
\sum_{j=1}^m \sum_{k = 1}^{b_j} \frac{r_{j, k} (t)}{[1 - f_j (t) s]^k} = \sum_{n \ge 0} (a-t)^{\delta} q_n (t) s^n
\]
and  $q_n (a) \neq 0$ for all $n \gg 0$. Moreover, the limit of $\frac{q_n (a)}{M^n n^L}$ as $n \to \infty$  exists and is equal to a non-zero rational number.
More precisely,  let $h (s, t) \in \R[s, t]$ be the polynomial determined by
\begin{equation}
\label{eq:def h}
\sum_{j=1}^m \sum_{k = 1}^{b_j} \frac{r_{j, k} (t)}{[1 - f_j (t) s]^k}
=
\frac{h(s, t)}{\prod_{j=1}^m [1 - f_j (t) s]^{b_j}},
\end{equation}
and write
$h (s, t) = (a - t)^{\delta} \cdot  \widetilde{h} (s, t)$,
where $\widetilde{h} (s, t) \in \R[s, t]$
such that
\begin{equation}
    \label{eq:notation hn}
0 \neq \widetilde{h} (s, a) = \sum_{n = 0}^N h_n s^n,
\text{ where } h_n\in \R.
\end{equation}
Set $b = b_1 + \cdots b_m$. Then the polynomial
\begin{equation}
    \label{eq:def H}
\sum_{j = 0}^N \binom{y+b-1-j}{b-1} h_j M^{N-j}  = \sum_{j=0}^N H_j y^j \in \R[y]
\end{equation}
with $H_j \in \R$ is not the zero polynomial. Let $L$ be its degree. Then
\[
\lim_{n \to \infty} \frac{q_n (a)}{M^n n^L} = \frac{H_L}{M^N}.
\]
\end{lem}

\begin{proof}
Since each $r_{j, b_j} (t)$ is not zero, Equation \eqref{eq:def h} shows that $h(s, t)$ is not the zero polynomial. Hence, there is an integer $\delta \ge 0$ such that
\[
h (s, t) = (a - t)^{\delta} \cdot  \widetilde{h} (s, t) \quad \text{ and } \quad \widetilde{h} (s, a) \neq 0.
\]
By Lemma \ref{lem:pol coefficients}, there are polynomials $\widetilde{q}_n (t) \in \R[t]$ such that
\[
\frac{\widetilde{h}(s, t)}{\prod_{j=1}^m [1 - f_j (t) s]^{b_j}} = \widetilde{h}(s, t) \cdot \sum_{n \ge 0} \widetilde{q}_n (t) s^n.
\]
Sorting by powers of $s$, this can be rewritten as
\begin{equation}
    \label{eq:def G}
G(s, t) = \frac{\widetilde{h}(s, t)}{\prod_{j=1}^m [1 - f_j (t) s]^{b_j}} = \sum_{n \ge 0} q_n (t) s^n
\end{equation}
with polynomials $q_n (t) \in \R[t]$. We conclude that
\[
\frac{h(s, t)}{\prod_{j=1}^m [1 - f_j (t) s]^{b_j}} = (a-t)^{\delta} \cdot G(s, t) = \sum_{n \ge 0} (a-t)^{\delta} q_n (t) s^n.
\]
Moreover, Equation \eqref{eq:def G} implies that, for each $t_0 \in \R$, the function  $G(s, t_0)$ can be represented  as a power series  $\sum_{n \ge 0} q_n (t_0) s^n$. 

We use this for $t_0 = a$. Setting $b = b_1 + \cdots + b_m$, our assumption $f_j (a) = M$ gives
\begin{align*}
G(s, a) & = \frac{\widetilde{h}(s, a)}{\prod_{j=1}^m [1 - M s]^{b_j}} =  \frac{\widetilde{h}(s, a)}{[1 - M s]^b} \\
& = \widetilde{h}(s, t) \cdot \sum_{n \ge 0} \binom{n+b-1}{b-1} M^n s^n.
\end{align*}
With the notation introduced in Equation \eqref{eq:notation hn}, this becomes
\[
G(s, a) = \sum_{n \ge 0} \left [ \sum_{j = 0}^{\min \{n, N\}} \binom{n+b-1-j}{b-1} h_j M^{n-j} \right ] s^n.
\]
Comparing coefficients we conclude, for all $n \ge N$,
\begin{equation}
   \label{eq:qn1}
q_n (a) = \sum_{j = 0}^{N} \binom{n+b-1-j}{b-1} h_j M^{n-j} = M^{n-N} \cdot \sum_{j = 0}^{N} \binom{n+b-1-j}{b-1} h_j M^{N-j}.
\end{equation}
Note that, as elements in $\R(t)[s]$, the polynomials $h(s, t)$ and $\widetilde{h} (s, t)$ have the same degree. Thus, Equation \eqref{eq:def h} shows that $0 \neq G(s,a)$ is a proper rational function in $s$. Hence, there is some $n_0 \ge N$ such that $q_{n_0}(a) \neq 0$.

Consider now the polynomial $H(y) \in  \R[y]$ defined by
\[
H(y) = \sum_{j = 0}^N \binom{y+b-1-j}{b-1} h_j M^{N-j}
\]
Comparing with Equation \eqref{eq:qn1}, we conclude, for all $n \ge N$,
\[
q_n (a) = M^{n-N} H(n).
\]
Since $q_{n_0}(a) \neq 0$, the polynomial $H(y)$ is not trivial. It follows that it has only finitely zeros, which implies $H(n) \neq 0$ for all $n \gg 0$. This yields $q_n (a) \neq 0$ for all $n \gg 0$, as desired.

Denote by $L$ the degree of $H(y)$, that is, with the notation introduced in Equation \eqref{eq:def H}
\[
H (y) = \sum_{j=0}^L H_j y^j,
\]
where $H_L \neq 0$. It now follows
\[
\lim_{n \to \infty} \frac{q_n (a)}{M^n n^L} = \frac{1}{M^N} \cdot \lim_{n \to \infty} \left (\sum_{j=0}^L H_j n^{j - L} \right ) =
\frac{H_L}{M^N} \neq 0,
\]
as claimed.
\end{proof}

We also need the following observation:

\begin{lem}
   \label{lem:sum estimate}
Consider real numbers
\[
a_1 = a_2 = \cdots = a_{l+1} > a_{l+2} \ge \cdots \ge a_k > 0.
\]
Then
\[
\lim_{n \to \infty}   \frac{\sum_{n_1 + \cdot + n_k = n, \ n_j \in \N_0} a_1^{n_1} \cdots a_k^{n_k}}{a_1^n \cdot n^l} = \frac{a_1^{k-1-l}}{l! \cdot \prod_{j = l+2}^k (a_1 - a_j)}.
\]
\end{lem}

\begin{proof}
Consider the rational function
$
q(t) = \prod_{j=1}^k \frac{1}{1- a_j t}.
$
The geometric series yields the following representation as a power series:
\[
q(t) =  \prod_{j=1}^k \left ( \sum_{n \ge 0} a_j^n t^n \right ) = \sum_{n \ge 0}  \left ( \sum_{n_1 + \cdot + n_k = n, \ n_j \in \N_0} a_1^{n_1} \cdots a_k^{n_k} \right ) t^n.
\]
It follows
\[
\sum_{n_1 + \cdot + n_k = n, \ n_j \in \N_0} a_1^{n_1} \cdots a_k^{n_k}  = \frac{q^{(n)} (0)}{n!}.
\]
The partial fraction decomposition of $q(t)$ has the form
\[
q(t) = \frac{r}{(1-a_1 t)^{l+1}} + \cdots,
\]
where $r = \frac{a_1^{k-1-l}}{\prod_{j = l+2}^k (a_1 - a_j)}$. It yields  a formula for the derivatives of $q(t)$, which implies the claim. We leave the details to the reader.
\end{proof}


\end{document}